\documentclass[dvips]{amsart}

\usepackage[foot]{amsaddr}
\title[Non-degeneracy conditions]{Non-degeneracy conditions for \\ braided finite tensor categories}
\author[K.~Shimizu]{Kenichi Shimizu}
\email{kshimizu@shibaura-it.ac.jp}
\address{Department of Mathematical Sciences \\
  Shibaura Institute of Technology \\
  307 Fukasaku, Minuma-ku, Saitama-shi, Saitama 337-8570, Japan.}
\thanks{The author is supported by JSPS KAKENHI Grant Number JP16K17568}
\date{}

\usepackage{amsmath,amssymb,amscd}
\usepackage{bbm}
\usepackage{stmaryrd}
\usepackage[all,knot]{xy}
\usepackage{url}

\usepackage{amsthm}
\numberwithin{equation}{section}

\newtheorem{counter}{}[section]

\theoremstyle{definition}
\newtheorem{definition}         [counter]{Definition}

\newtheorem*{notation*}         {Notation}
\theoremstyle{plain}
\newtheorem{lemma}              [counter]{Lemma}

\newtheorem{proposition}        [counter]{Proposition}
\newtheorem{theorem}            [counter]{Theorem}

\newtheorem*{theorem*}          {Theorem}
\theoremstyle{remark}
\newtheorem{remark}             [counter]{Remark}
\newtheorem{example}            [counter]{Example}

\newcommand{\id}{\mathrm{id}}
\newcommand{\eval}{{\rm ev}}
\newcommand{\coev}{{\rm coev}}
\newcommand{\op}{\mathrm{op}}
\newcommand{\rev}{\mathrm{rev}}

\newcommand{\unitobj}{\mathbbm{1}}

\newcommand{\Hom}{\mathrm{Hom}}
\newcommand{\End}{\mathrm{End}}

\newcommand{\iHom}{\underline{\mathrm{Hom}}}

\newcommand{\LEX}{\textsc{Lex}}

\newcommand{\FPdim}{\operatorname{FPdim}}
\newcommand{\Mueg}{\text{\rm M\"ug}}
\newcommand{\Rep}{\operatorname{Rep}}

\newenvironment{xy-picture}{\xy}{\endxy}

\newcommand{\mycap}[1]{\POS{:
    (0,0); (0,#1) **@{} ?(.5)="_P1",
    (1,0); (1,#1) **@{} ?(.5)="_P2",
    (0,0); (1,0) **\crv{"_P1" & (0.5,#1) & "_P2"}}}

\newcommand{\mybendWithOpt}[2]{\POS{
    **@{} ?<="_P1" ?>="_P2",
    {"_P1"; p-(0,1) **@{} ?!{"_P2"; p-(1,0)}}="_P3",
    {"_P2"; p-(0,1) **@{} ?!{"_P1"; p-(1,0)}}="_P4",
    "_P1"; "_P3" **@{} ?(#1)="_P5",
    "_P2"; "_P4" **@{} ?(#1)="_P6",
    "_P1"; "_P2" **\crv{#2 "_P5" & "_P6"}
}}

\newcommand{\mybend}[1]{\POS{\mybendWithOpt{#1}{}}}

\newcommand{\mytrident}[3]{\POS{
    c; p-(#1,0) **@{} ?(.5)="_P1",
    "_P1"; p+(#1,-#3)+(#2,0) \mybend{.8}
    ?>; p+(#1,0) **\dir{-} ?(.5) @+ 
    ?>; p+(#1,#3)+(#2,0) \mybend{.8}
    ?>; p-(#1,0) **\dir{-} ?(.5) @+ 
    ?>; p-(#2,0) \mycap{1.25}
    ?>; p-(#1,0) **\dir{-} ?(.5) @+ 
    ?>; p-(#2,0) \mycap{1.25}
    ?>; p-(#1,0) **\dir{-} ?(.5) @+ 
}}

\newcommand{\pushrect}[2]{\POS{ c="_P1",
    "_P1"+(#1,-#2) @+, "_P1"+( 0,-#2) @+,
    "_P1"+(#1,  0) @+, "_P1"+( 0,  0) @+ }}

\newcommand{\makecoord}[2]{\POS{#1; p+(0, 1) **\dir{} ?!{#2; p+(1, 0)}}}

\begin{document}

\maketitle

\begin{abstract}
  For a braided finite tensor category $\mathcal{C}$ with unit object $1 \in \mathcal{C}$, Lyubashenko considered a certain Hopf algebra $\mathbb{F} \in \mathcal{C}$ endowed with a Hopf pairing $\omega: \mathbb{F} \otimes \mathbb{F} \to 1$ to define the notion of a `non-semisimple' modular tensor category. We say that $\mathcal{C}$ is non-degenerate if the Hopf pairing $\omega$ is non-degenerate. In this paper, we show that $\mathcal{C}$ is non-degenerate if and only if it is factorizable in the sense of Etingof, Nikshych and Ostrik, if and only if its M\"uger center is trivial, if and only if the linear map $\mathrm{Hom}_{\mathcal{C}}(1, \mathbb{F}) \to \mathrm{Hom}_{\mathcal{C}}(\mathbb{F}, 1)$ induced by the pairing $\omega$ is invertible. As an application, we prove that the category of Yetter-Drinfeld modules over a Hopf algebra in $\mathcal{C}$ is non-degenerate if and only if $\mathcal{C}$ is.
\end{abstract}

\section{Introduction}

The $S$-matrix of a ribbon fusion category $\mathcal{C}$ is the square matrix whose $(i, j)$-th entry is the invariant of the Hopf link colored with $i$ and $j$, where $i$ and $j$ run over the isomorphism classes of simple objects of $\mathcal{C}$. A {\em modular tensor category} is a ribbon fusion category whose $S$-matrix is invertible. This notion is widely studied in connection with conformal field theories, topological quantum field theories and quantum computing; see, {\it e.g.}, \cite{MR1292673,MR1797619,MR3242743} and references therein.

With motivation coming from conformal field theories and topological quantum field theories \cite{MR3146014}, and also from purely mathematical point of view, it is interesting to consider a `non-semisimple' generalization of the notion of a modular tensor category. Such a notion has been proposed and investigated by Lyubashenko \cite{MR1324034,MR1352517,MR1324033}: If $\mathcal{C}$ is a braided finite tensor category (which is not necessarily semisimple), then the coend
\begin{equation}
  \mathbb{F} = \int^{X \in \mathcal{C}} X \otimes X^*
\end{equation}
is a Hopf algebra in $\mathcal{C}$ and has the Hopf pairing $\omega_{\mathcal{C}}: \mathbb{F} \otimes \mathbb{F} \to \unitobj$ defined in terms of the braiding of $\mathcal{C}$. We say that $\mathcal{C}$ is {\em non-degenerate} if $\omega_{\mathcal{C}}$ is. Kerler and Lyubashenko \cite{MR1862634} used the term `modular tensor category' to mean a non-degenerate ribbon finite tensor category.

As explained in \cite{MR1862634}, a `non-semisimple' modular tensor category in this sense also yields an invariant of closed 3-manifolds and a projective representation of the mapping class group of a closed surface as in the semisimple case. On the other hand, there seems to be a difficulty in dealing with and analyzing the coend $\mathbb{F}$ and thus it is not easy to check whether a given braided finite tensor category is non-degenerate. The aim of this paper is to give conditions for a braided finite tensor category that are equivalent to the non-degeneracy.

Before describing our results, we first explain how the Hopf pairing $\omega_{\mathcal{C}}$ relates to the $S$-matrix in the semisimple case. We consider the map
\begin{equation*}
  \Omega_{\mathcal{C}}: \Hom_{\mathcal{C}}(\unitobj, \mathbb{F}) \to \Hom_{\mathcal{C}}(\mathbb{F}, \unitobj),
  \quad f \mapsto (f \otimes \id_{\mathbb{F}}) \circ \omega_{\mathcal{C}}
\end{equation*}
induced by $\omega_{\mathcal{C}}$. Following Takeuchi \cite{MR1850651}, we say that $\mathcal{C}$ is {\em weakly-factorizable} if $\Omega_{\mathcal{C}}$ is invertible. If $\mathcal{C}$ is a ribbon fusion category, the source and the target of $\Omega_{\mathcal{C}}$ have natural bases, and the $S$-matrix is in fact the representation matrix of $\Omega_{\mathcal{C}}$ with respect to these bases (see \S\ref{subsec:vs-S-matrix} for details). Thus the $S$-matrix of $\mathcal{C}$ is invertible if and only if $\mathcal{C}$ is weakly-factorizable.

Next, we recall that the non-degeneracy of a braided finite tensor category $\mathcal{C}$ is a generalization of the factorizability of a Hopf algebra \cite{MR1075721}; see \cite[\S7.4.6]{MR1862634} for the detail. On the other hand, Etingof, Nikshych and Ostrik \cite{MR2097289} introduced a certain functor $G: \mathcal{C} \boxtimes \mathcal{C} \to \mathcal{Z}(\mathcal{C})$ defined in terms of the braiding (see \S\ref{subsec:factorizability}), and then defined that $\mathcal{C}$ is {\em factorizable} if $G$ is an equivalence. This condition is based on Schneider's characterization of factorizability of finite-dimensional quasitriangular Hopf algebras \cite[Theorem 4.3]{MR1825894}.

Finally, we introduce the following notion: For a full subcategory $\mathcal{D}$ of $\mathcal{C}$, the {\em M\"uger centralizer} \cite{MR1966525} of $\mathcal{D}$ in $\mathcal{C}$, denoted by $\mathcal{D}'$, is defined to be the full subcategory of $\mathcal{C}$ consisting of all objects $X \in \mathcal{C}$ such that $\sigma_{Y,X} \circ \sigma_{X,Y} = \id_{X \otimes Y}$ for all $Y \in \mathcal{D}$, where $\sigma$ is the braiding of $\mathcal{C}$. We call $\mathcal{C}'$ {\em the M\"uger center} of $\mathcal{C}$, and say that {\em the M\"uger center of $\mathcal{C}$ is trivial} if every object of $\mathcal{C}'$ is isomorphic to the direct sum of finitely many copies of the unit object $\unitobj \in \mathcal{C}$.

Now our results in this paper are summarized as follows:

\begin{theorem}
  \label{thm:main}
  For a braided finite tensor category $\mathcal{C}$, the following assertions are equivalent:
  \begin{enumerate}
  \item $\mathcal{C}$ is non-degenerate.
  \item $\mathcal{C}$ is factorizable.
  \item $\mathcal{C}$ is weakly-factorizable.
  \item The M\"uger center of $\mathcal{C}$ is trivial.
  \end{enumerate}
\end{theorem}

This theorem follows from Theorems~\ref{thm:factor-nondege},~\ref{thm:factor-tricen} and \ref{thm:inj-tricen}. It has been known that these conditions are equivalent in the case where $\mathcal{C}$ is semisimple; see \cite{MR1741269,MR1966525,MR2609644,MR3242743}.
We therefore have obtained a non-semisimple generalization of characterizations of the non-degeneracy of a braided fusion category.

For a Hopf algebra $B \in \mathcal{C}$, the category $\mathcal{YD}(\mathcal{C})^B_B$ of Yetter-Drinfeld modules over $B$ is defined \cite{MR1456522}. The category $\mathcal{YD}(\mathcal{C})^B_B$ is in fact a braided finite tensor category.
As an application of the above theorem, we prove that $\mathcal{YD}(\mathcal{C})^B_B$ is non-degenerate if and only if $\mathcal{C}$ is (Theorem~\ref{thm:YD-br-FTC}). As we will explain in \S\ref{sec:fact-hopf-alg}, this result provides a new source of factorizable Hopf algebras which are not the Drinfeld double in general.

\subsection*{Organization of this paper}

The present paper is organized as follows: In Section~\ref{sec:preliminaries}, we fix conventions and recall basic results on finite tensor categories, Frobenius-Perron dimensions, and Hopf monads from \cite{MR1712872,MR1321145,MR3242743,MR2355605,MR2869176,MR2793022}.

In Section~\ref{sec:factor-vs-non-dege}, we prove the equivalence $(1) \Leftrightarrow (2)$ of the above theorem (Theorem~\ref{thm:factor-nondege}). Till the end of this introduction, we assume that $\mathcal{C}$ is a braided finite tensor category. We first recall the definitions of the non-degeneracy and the factorizability of $\mathcal{C}$ in detail. It turns out that the category $\mathcal{C}^{\mathbb{F}}$ of $\mathbb{F}$-comodules is equivalent to $\mathcal{C} \boxtimes \mathcal{C}$, and the category $\mathcal{C}_{\mathbb{F}}$ of $\mathbb{F}$-modules is equivalent to $\mathcal{Z}(\mathcal{C})$ (Lemmas~\ref{lem:F-comodules} and~\ref{lem:F-modules}). Let $\omega^{\natural}: \mathcal{C}^{\mathbb{F}} \to \mathcal{C}_{\mathbb{F}}$ be the functor induced by the Hopf pairing $\omega = \omega_{\mathcal{C}}$. We prove that the functor obtained by the composition
\begin{equation*}
  \mathcal{C} \boxtimes \mathcal{C}
  \xrightarrow{\quad \approx \quad}
  \mathcal{C}^{\mathbb{F}}
  \xrightarrow{\quad \omega^{\natural} \quad}
  \mathcal{C}_{\mathbb{F}}
  \xrightarrow{\quad \approx \quad}
  \mathcal{Z}(\mathcal{C})
\end{equation*}
is isomorphic to the functor $\mathcal{C} \boxtimes \mathcal{C} \to \mathcal{Z}(\mathcal{C})$ used to define the factorizability. The equivalence $(1) \Leftrightarrow (2)$ is easily proved once the above claim is verified.

In Section~\ref{sec:factor-vs-tricen}, we prove the equivalence $(2) \Leftrightarrow (4)$ of the above theorem (Theorem~\ref{thm:factor-tricen}). By using the central Hopf monad and its relation to the induction to the Drinfeld center, we prove the following formula for the Frobenius-Perron dimensions: If $\mathcal{A}$ and $\mathcal{B}$ are tensor full subcategories of $\mathcal{C}$ (see Definition \ref{def:ten-ful-subcat} for the precise meaning), then we have
\begin{equation*}
  \FPdim(\mathcal{A} \cap \mathcal{B})
  \FPdim(\mathcal{A} \vee \mathcal{B})
  = \FPdim(\mathcal{A}) \FPdim(\mathcal{B}),
\end{equation*}
where $\mathcal{A} \cap \mathcal{B}$ is the intersection of $\mathcal{A}$ and $\mathcal{B}$ and $\mathcal{A} \vee \mathcal{B}$ is the tensor full subcategory of $\mathcal{C}$ `generated' by them (Lemma~\ref{lem:FPdim-cap}). The equivalence $(2) \Leftrightarrow (4)$ is proved by applying this formula to two particular copies of $\mathcal{C}$ in $\mathcal{Z}(\mathcal{C})$.
The formula also yields the equations
\begin{equation*}
  \FPdim(\mathcal{D}) \FPdim(\mathcal{D}') = \FPdim(\mathcal{C}) \FPdim(\mathcal{D} \cap \mathcal{C}')
  \quad \text{and} \quad
  \mathcal{D}'' = \mathcal{D} \vee \mathcal{C}'
\end{equation*}
for a tensor full subcategory $\mathcal{D} \subset \mathcal{C}$ (Theorem~\ref{thm:dim-muger-cent}), which are already known in the semisimple case \cite[\S8.21]{MR3242743}.

Note that the implication $(1) \Rightarrow (3)$ of Theorem~\ref{thm:main} is obvious. In Section~\ref{sec:weak-factori}, we prove $(3) \Rightarrow (4)$ to complete the proof of the theorem. More precisely, we show that the M\"uger center $\mathcal{C}'$ is trivial if $\Omega_{\mathcal{C}}$ is injective (Theorem~\ref{thm:inj-tricen}). For this purpose, we consider the coend
\begin{equation*}
  \mathbb{F}' = \int^{X \in \mathcal{C}'} X^* \otimes X
\end{equation*}
whose `domain of integration' is different from $\mathbb{F}$. As we have shown in \cite{2015arXiv150401178S}, there is a canonical monomorphism $\phi: \mathbb{F}' \to \mathbb{F}$ that respects the universal dinatural transformations. By the definition of $\Omega_{\mathcal{C}}$, we see that the image of the map
\begin{equation*}
  \Hom_{\mathcal{C}}(\unitobj, \mathbb{F}')
  \xrightarrow{\quad \Hom_{\mathcal{C}}(\unitobj, \phi) \quad}
  \Hom_{\mathcal{C}}(\unitobj, \mathbb{F})
  \xrightarrow{\quad \Omega_{\mathcal{C}} \quad}
  \Hom_{\mathcal{C}}(\mathbb{F}, \unitobj)
\end{equation*}
is one-dimensional. Hence, if $\Omega_{\mathcal{C}}$ is injective, then we have
\begin{equation}
  \label{eq:intro-5-3}
  \dim_k \Hom_{\mathcal{C}}(\unitobj, \mathbb{F}') = 1
\end{equation}
As an application of the integral theory for unimodular finite tensor categories developed in \cite{2015arXiv150401178S}, we prove that \eqref{eq:intro-5-3} is equivalent to that the M\"uger center of $\mathcal{C}$ is trivial. We also give some observations on the rank of $\Omega_{\mathcal{C}}$.

In Section~\ref{sec:non-dege-YD}, we give an application of Theorem~\ref{thm:main}. Let $B$ be a Hopf algebra in $\mathcal{C}$. We first observe that the M\"uger center $\mathcal{C}'$ can be embedded into the category $\mathcal{YD}(\mathcal{C})^B_B$ of Yetter-Drinfeld modules. The main result of this section (Theorem~\ref{thm:YD-br-FTC}) states that $\mathcal{YD}(\mathcal{C})^B_B$ is a braided finite tensor category such that
\begin{equation*}
  \FPdim(\mathcal{YD}(\mathcal{C})^B_B) = \FPdim(\mathcal{C}) \FPdim(B)^2,
\end{equation*}
and the M\"uger center of $\mathcal{YD}(\mathcal{C})^B_B$ is precisely the category $\mathcal{C}'$. Thus, by our result, $\mathcal{YD}(\mathcal{C})^B_B$ is non-degenerate if and only if $\mathcal{C}$ is. Finally, in \S\ref{sec:fact-hopf-alg}, we explain how this result yields examples of factorizable Hopf algebras including so-called small quantum groups.

\subsection*{Acknowledgments}

The author is supported by JSPS KAKENHI Grant Number JP16K17568.

\section{Preliminaries}
\label{sec:preliminaries}

\subsection{Monoidal categories}

For the basic theory of monoidal categories, we refer the reader to \cite{MR1321145} and \cite{MR3242743}. All monoidal categories are assumed to be strict. Given a monoidal category $\mathcal{C} = (\mathcal{C}, \otimes, \unitobj)$ with tensor product $\otimes$ and unit object $\unitobj$, we set $\mathcal{C}^{\op} = (\mathcal{C}^{\op}, \otimes, \unitobj)$ and $\mathcal{C}^{\rev} = (\mathcal{C}, \otimes^{\rev}, \unitobj)$, where $(-)^{\op}$ means the opposite category and $\otimes^{\rev}$ is the reversed tensor product defined by $X \otimes^{\rev} Y = Y \otimes X$.

Our notation for duality follows \cite[\S2.10]{MR3242743}. Thus, for an object $X$ of a rigid monoidal category $\mathcal{C}$, we denote by $X^*$ the left dual object of $X$ with evaluation $\eval_X: X^* \otimes X \to \unitobj$ and coevaluation $\coev_X: \unitobj \to X \otimes X^*$. If $\mathcal{C}$ is rigid, then the assignment $X \mapsto X^*$ gives rise to an equivalence $(-)^*: \mathcal{C}^{\op} \to \mathcal{C}^{\rev}$ of monoidal categories. A quasi-inverse of $(-)^*$, denoted by ${}^*(-)$, is given by taking a right dual object. For simplicity, we assume that $(-)^*$ and ${}^*(-)$ are strict monoidal functors and mutually inverse to each other.

\subsection{Finite tensor categories}

Throughout, we work over an algebraically closed field $k$. Given an algebra $A$ over $k$, we denote by $\Rep(A)$ the category of finite-dimensional left $A$-modules. In particular, $\mathrm{Vec} := \Rep(k)$ is the category of finite-dimensional vector spaces. A {\em finite abelian category} is a $k$-linear category that is equivalent to $\Rep(A)$ for some finite-dimensional algebra $A$ over $k$. We note that, by the Eilenberg-Watts theorem, a $k$-linear functor between finite abelian categories has a left (right) adjoint if and only if it is left (right) exact.

\begin{lemma}
  \label{lem:finiteness-lemma}
  Let $\mathcal{A}$ be a finite abelian category over $k$, and let $T$ be a $k$-linear left exact comonad on $\mathcal{A}$. Then the category $\mathcal{A}^T$ of $T$-comodules is also a finite abelian category over $k$ such that the forgetful functor $U: \mathcal{A}^T \to \mathcal{A}$ preserves and reflects exact sequences.
\end{lemma}
\begin{proof}
  By the dual of \cite[Proposition 5.3]{MR0184984}, the category $\mathcal{A}^T$ is a $k$-linear abelian category such that $U$ preserves and reflects exact sequences. Now let $P$ be a projective generator of $\mathcal{A}^{\op}$ (which exists since also $\mathcal{A}^{\op}$ is finite). Then, since
  \begin{equation*}
    \Hom_{\mathcal{A}^T}(-, T(P)) \cong \Hom_{\mathcal{A}}(U(-), P) = \Hom_{\mathcal{A}^{\op}}(P, -) \circ U,
  \end{equation*}
  the $T$-comodule $T(P)$ is a projective generator of $(\mathcal{A}^T)^{\op}$. By the standard argument, one can show that $(\mathcal{A}^T)^{\op}$ is a finite abelian category. Hence so is $\mathcal{A}^T$.
\end{proof}

A {\em finite tensor category} \cite{MR2119143} is a rigid monoidal category such that $\mathcal{C}$ is a finite abelian category, the tensor product of $\mathcal{C}$ is $k$-linear in each variable, and the unit object of $\mathcal{C}$ is a simple object. By a {\em tensor functor}, we mean a $k$-linear exact strong monoidal functor between finite tensor categories.

Let $\mathcal{A}$ and $\mathcal{B}$ be $k$-linear categories. A $k$-linear functor $F: \mathcal{A} \to \mathcal{B}$ is said to be {\em dominant} ($=$ surjective \cite{MR3242743}) if every object of $\mathcal{B}$ is a subobject of $F(X)$ for some $X \in \mathcal{A}$. Suppose that $\mathcal{A}$ and $\mathcal{B}$ are finite tensor categories. A tensor functor $F: \mathcal{A} \to \mathcal{B}$ is dominant if and only if every object of $\mathcal{D}$ is a quotient of $F(X)$ for some $X \in \mathcal{C}$, if and only if a left adjoint of $F$ is faithful, if and only if a right adjoint of $F$ is faithful \cite[Lemma 3.1]{MR2863377}.

\subsection{Frobenius-Perron dimension}

Let $\mathcal{C}$ be a finite tensor category. For $X \in \mathcal{C}$, we denote by $\FPdim(X) \in \mathbb{R}_{+}$ the {\em Frobenius-Perron dimension} of $X$ \cite{MR2119143}. The Frobenius-Perron dimension of $\mathcal{C}$ is defined by
\begin{equation*}
  \FPdim(\mathcal{C}) := \sum_{i = 0}^m \FPdim(V_i) \FPdim(P_i),
\end{equation*}
where $\{ V_i \}_{i = 0}^m$ is the complete set of representatives of isomorphism classes of simple objects of $\mathcal{C}$ and $P_i$ is the projective cover of $V_i$.

For the basic properties of the Frobenius-Perron dimensions, we refer the reader to \cite[Chapter 6]{MR3242743}. We recall the following useful formula which is important in this paper: Let $F: \mathcal{A} \to \mathcal{B}$ be a dominant tensor functor between finite tensor categories. If $I$ is right adjoint to $F$, then we have
\begin{equation}
  \label{eq:FPdim-ind}
  \FPdim(I(X)) = \frac{\FPdim(\mathcal{A})}{\FPdim(\mathcal{B})} \FPdim(X)
\end{equation}
for all $X \in \mathcal{B}$ \cite[Lemma 6.2.4]{MR3242743}. This formula holds also in the case where $I$ is {\em left} adjoint to $F$, since then $X \mapsto {}^*I(X^*)$ is {\em right} adjoint to $F$.

\subsection{Ends and coends}

Let $\mathcal{A}$ and $\mathcal{V}$ be categories, and let $P$ and $Q$ be functors from $\mathcal{A}^{\op} \times \mathcal{A}$ to $\mathcal{V}$. A {\em dinatural transformation} $\xi$ from $P$ to $Q$ is a family
\begin{equation*}
  \xi = \{ \xi_X: P(X, X) \to Q(X, X) \}_{X \in \mathcal{A}}
\end{equation*}
of morphisms in $\mathcal{V}$ satisfying
\begin{equation*}
  Q(X, f) \circ \xi_X \circ P(f, X)
  = Q(f, Y) \circ \xi_Y \circ P(Y, f)
\end{equation*}
for all morphisms $f: X \to Y$ in $\mathcal{A}$. For a while, we regard an object $M \in \mathcal{V}$ as a constant functor from $\mathcal{A}^{\op} \times \mathcal{A}$ to $\mathcal{V}$. Then an {\em end} of $Q$ is an object $E \in \mathcal{V}$ endowed with a dinatural transformation $\xi$ from $E$ to $Q$ satisfying the following universal property: For every dinatural transformation $\xi'$ from $E' \in \mathcal{V}$ to $Q$, there exists a unique morphism $\phi: E' \to E$ such that $\xi'_X = \xi_X \circ \phi$ for all $X \in \mathcal{A}$. A {\em coend} of $P$ is an object $C \in \mathcal{V}$ endowed with a dinatural transformation from $P$ to $C$ satisfying a similar universal property. An end of $Q$ and a coend of $P$ are written as
\begin{equation*}
  \int_{X \in \mathcal{A}} Q(X, X)
  \quad \text{and} \quad
  \int^{X \in \mathcal{A}} P(X, X),
\end{equation*}
respectively. See \cite[IX]{MR1712872} for basic properties of ends and coends. We also note the following lemma \cite[Lemma 3.9]{MR2869176}:

\begin{lemma}
  \label{lem:coend-adj}
  Let $\mathcal{A}$, $\mathcal{B}$ and $\mathcal{V}$ be categories, let $G: \mathcal{A} \to \mathcal{B}$ be a functor with left adjoint $F: \mathcal{B} \to \mathcal{A}$, and let $P: \mathcal{A}^{\op} \times \mathcal{B} \to \mathcal{V}$ be a functor. Then we have
  \begin{equation*}
    \int^{X \in \mathcal{A}} P(X, G(X)) \cong \int^{X \in \mathcal{B}} P(F(X), X),
  \end{equation*}
  meaning that if either one of these coends exists, then both exist and they are canonically isomorphic.
\end{lemma}

\subsection{Deligne tensor product}

The Deligne tensor product $\mathcal{A} \boxtimes \mathcal{B}$ of finite abelian categories $\mathcal{A}$ and $\mathcal{B}$ is a finite abelian category having a certain universal property for functors from $\mathcal{A} \times \mathcal{B}$ that are $k$-linear and right exact in each variable. If $\mathcal{A}$ and $\mathcal{B}$ are (braided) finite tensor categories, then $\mathcal{A} \boxtimes \mathcal{B}$ is naturally a (braided) finite tensor category; see \cite{MR1106898} or \cite[\S1.11]{MR3242743}.

For finite abelian categories $\mathcal{A}$ and $\mathcal{B}$, we denote by $\LEX(\mathcal{A}, \mathcal{B})$ the category of $k$-linear left exact functors from $\mathcal{A}$ to $\mathcal{B}$. Recall that the Deligne tensor product $\mathcal{A} \boxtimes \mathcal{B}$ also has a universal property for functors from $\mathcal{A} \times \mathcal{B}$ that are $k$-linear and {\em left} exact in each variable. Using this universal property, we define a $k$-linear {\em left} exact functor
\begin{equation}
  \label{eq:def-Phi}
  \Phi_{\mathcal{A}, \mathcal{B}}:
  \mathcal{A}^{\op} \boxtimes \mathcal{B} \to \LEX(\mathcal{A}, \mathcal{B}),
  \quad X \boxtimes Y \mapsto \Hom_{\mathcal{A}}(X, -) \cdot Y,
\end{equation}
where ``$\cdot$'' means the canonical action of $\mathrm{Vec}$ on $\mathcal{B}$. The Eilenberg-Watts theorem implies that $\Phi_{\mathcal{A}, \mathcal{B}}$ is an equivalence. Moreover, the coend
\begin{equation}
  \label{eq:def-Phi-inverse}
  \overline{\Phi}_{\mathcal{A}, \mathcal{B}}(F) = \int^{X \in \mathcal{A}} \! X \boxtimes F(X)
\end{equation}
exists for all $F \in \LEX(\mathcal{A}, \mathcal{B})$, and the assignment $F \mapsto \overline{\Phi}_{\mathcal{A}, \mathcal{B}}(F)$ is in fact a quasi-inverse of the equivalence $\Phi_{\mathcal{A}, \mathcal{B}}$ (see \cite{2014arXiv1402.3482S} for the detail).

Now let $\mathcal{C}$ be a finite tensor category, and let $D = (-)^*$ be the left duality functor on $\mathcal{C}$. By considering the image of $F \in \LEX(\mathcal{C}, \mathcal{C})$ under the functor
\begin{equation*}
  \LEX(\mathcal{C}, \mathcal{C})
  \xrightarrow{\quad \overline{\Phi} \quad}
  \mathcal{C}^{\op} \boxtimes \mathcal{C}
  \xrightarrow{\quad D \boxtimes \id \quad}
  \mathcal{C} \boxtimes \mathcal{C}
  \xrightarrow{\quad \otimes \quad} \mathcal{C},
\end{equation*}
we obtain the following result:

\begin{lemma}
  \label{lem:coend-exist}
  The coend $\displaystyle \int^{X \in \mathcal{C}} \! X^* \otimes F(X)$ exists for each $F \in \LEX(\mathcal{C}, \mathcal{C})$.
\end{lemma}

\subsection{The Drinfeld center}
\label{subsec:dri-cen}

Let $\mathcal{C}$ be a monoidal category, and let $\mathcal{S}$ be a monoidal full subcategory of $\mathcal{C}$. The centralizer of $\mathcal{S}$ is the monoidal category $\mathcal{Z}(\mathcal{S}; \mathcal{C})$ defined as follows: An object of this category is a pair $(V, c)$ consisting of an object $V$ of $\mathcal{C}$ and a natural isomorphism $c_{X}: V \otimes X \to X \otimes V$ ($X \in \mathcal{S}$) satisfying
\begin{equation*}
  c_{X \otimes Y} = (\id_X \otimes c_Y) \circ (c_{X} \otimes \id_Y)
\end{equation*}
for all objects $X, Y \in \mathcal{S}$. A morphism $f: (V, c) \to (W, d)$ in $\mathcal{Z}(\mathcal{S}; \mathcal{C})$ is a morphism $f: V \to W$ in $\mathcal{C}$ such that $(\id_X \otimes f) \circ c_X = d_X \circ (f \otimes \id_X)$ for all $X \in \mathcal{S}$. The tensor product of objects of $\mathcal{Z}(\mathcal{S}; \mathcal{C})$ is given by the formula
\begin{equation*}
  (V, c) \otimes (W, d) = (V \otimes W, (c \otimes \id_{W})(\id_V \otimes d))
\end{equation*}
for $(V, c), (W, d) \in \mathcal{Z}(\mathcal{S}; \mathcal{C})$. The composition and the tensor product of morphisms are defined in an obvious way.

\begin{definition}
  We call $\mathcal{Z}(\mathcal{C}) := \mathcal{Z}(\mathcal{C}; \mathcal{C})$ the {\em Drinfeld center} of $\mathcal{C}$.
\end{definition}

The Drinfeld center has the braiding $\Sigma$ given by $\Sigma_{(V,c), (W, d)} = c_{W}$.

\subsection{The central Hopf monad}
\label{subsec:Hopf-monad-Z}

Let $\mathcal{C}$ be a finite tensor category. A {\em bimonad} on $\mathcal{C}$ is a comonoidal ($=$ oplax monoidal) endofunctor on $\mathcal{C}$ endowed with a structure of a monad such that the multiplication and the unit are comonoidal natural transformations. A {\em Hopf monad} on $\mathcal{C}$ is a bimonad on $\mathcal{C}$ admitting an antipode; see \cite{MR2355605,MR2869176,MR2793022} for the basic theory of Hopf monads.

The quantum double of the identity Hopf monad, which we call the central Hopf monad, plays an important role in this paper. It is defined as follows: Applying Lemma~\ref{lem:coend-exist} to the functor $F = V \otimes (-)$, we see that the coend
\begin{equation}
  Z_{\mathcal{C}}(V) = \int^{X \in \mathcal{C}} X^* \otimes V \otimes X
\end{equation}
exists for each $V \in \mathcal{C}$. Let $i_{\mathcal{C}}(V; X): X^* \otimes V \otimes X \to Z_{\mathcal{C}}(V)$ ($V, X \in \mathcal{C}$) denote the universal dinatural transformation for the coend. By the parameter theorem for coends \cite[IX.7]{MR1712872}, the assignment $V \mapsto Z_{\mathcal{C}}(V)$ extends to an endofunctor on $\mathcal{C}$ such that the morphism $i_{\mathcal{C}}(V; X)$ is natural in $V$. We now define
\begin{equation*}
  Z_{\mathcal{C}}^{(0)}:
  Z_{\mathcal{C}}(\unitobj) \to \unitobj
  \quad \text{and} \quad
  Z_{\mathcal{C}}^{(2)}(V, W):
  Z_{\mathcal{C}}(V \otimes W) \to Z_{\mathcal{C}}(V) \otimes Z_{\mathcal{C}}(W)
\end{equation*}
for $V, W \in \mathcal{C}$ to be the unique morphisms such that $Z_{\mathcal{C}}^{(0)} \circ i(\unitobj; X) = \eval_X$ and
\begin{equation}
  \label{eq:Hopf-monad-Z-def-Z2}
  \begin{aligned}
    & Z_{\mathcal{C}}^{(2)}(V, W) \circ i_{\mathcal{C}}(V \otimes W; X) \\
    & \quad = (i_{\mathcal{C}}(V; X) \otimes i_{\mathcal{C}}(W; X))
    \circ (\id_{X^*} \otimes \id_{V} \otimes \coev_X \otimes \id_{W} \otimes \id_{X})    
  \end{aligned}
\end{equation}
for all $X \in \mathcal{C}$. Then $Z_{\mathcal{C}} = (Z_{\mathcal{C}}, Z_{\mathcal{C}}^{(2)}, Z_{\mathcal{C}}^{(0)})$ is a comonoidal endofunctor on $\mathcal{C}$. With the help of the Fubini theorem for coends \cite[IX.8]{MR1712872}, we define $\mu: Z_{\mathcal{C}} \circ Z_{\mathcal{C}} \to Z_{\mathcal{C}}$ to be the unique morphism such that
\begin{equation}
  \label{eq:Hopf-monad-Z-def-mu}
  \mu_{V} \circ i_{\mathcal{C}}(Z_{\mathcal{C}}(V); Y) \circ (\id_{Y^*} \otimes i_{\mathcal{C}}(V; X) \otimes \id_Y)
  = i_{\mathcal{C}}(V; X \otimes Y)
\end{equation}
for $V, X, Y \in \mathcal{C}$. Finally, we define $\eta: \id_{\mathcal{C}} \to Z_{\mathcal{C}}$ by $\eta_V = i_{\mathcal{C}}(V; \unitobj)$ for $V \in \mathcal{C}$. Then the triple $(Z_{\mathcal{C}}, \mu, \eta)$ is a monad on $\mathcal{C}$, which is in fact a quasitriangular Hopf monad on $\mathcal{C}$ (we omit the description of other structure morphisms since we will not use them; see \cite{MR2869176} for details).

\begin{definition}
  We call $Z_{\mathcal{C}}$ the {\em central Hopf monad} on $\mathcal{C}$.
\end{definition}

Given an object $(V, c) \in \mathcal{Z}(\mathcal{C})$, we define $a: Z_{\mathcal{C}}(V) \to V$ in $\mathcal{C}$ by
\begin{equation}
  \label{eq:center-to-Z-mod}
  a \circ i_{\mathcal{C}}(V; X) = (\eval_X \otimes \id_V) \circ (\id_{X^*} \otimes c_X)
\end{equation}
for $X \in \mathcal{C}$. Then the pair $(V, a)$ is a $Z_{\mathcal{C}}$-module. This correspondence allows us to identify $\mathcal{Z}(\mathcal{C})$ with the category of $Z_{\mathcal{C}}$-modules.

\section{Non-degeneracy and factorizability}
\label{sec:factor-vs-non-dege}

\subsection{Non-degeneracy}
\label{subsec:non-dege}

Let $\mathcal{C}$ be a braided finite tensor category with braiding $\sigma$, and let $Z_{\mathcal{C}}$ be the central Hopf monad on $\mathcal{C}$. Then $\mathbb{F} := Z_{\mathcal{C}}(\unitobj)$ is a coalgebra with the comultiplication $\Delta$ and the counit $\varepsilon$ given respectively by
\begin{equation}
  \label{eq:F-def-coalg}
  \Delta = Z_{\mathcal{C}}^{(2)}(\unitobj, \unitobj)
  \quad \text{and} \quad
  \varepsilon = Z_{\mathcal{C}}^{(0)}.
\end{equation}
The coalgebra $\mathbb{F}$ is in fact a Hopf algebra in $\mathcal{C}$ with structure morphisms defined as follows: The multiplication is the unique morphism $m: \mathbb{F} \otimes \mathbb{F} \to \mathbb{F}$ such that
\begin{equation}
  \label{eq:F-def-mult}
  m \circ (i(\unitobj; X) \otimes i(\unitobj; Y))
  = i(\unitobj; X \otimes Y) \circ (\sigma_{X^* \otimes X, Y^*} \otimes Y)
\end{equation}
for all $X, Y \in \mathcal{C}$. The unit $u: \unitobj \to \mathbb{F}$ is given by $u = i(\unitobj)$. The antipode of $\mathbb{F}$ is the unique morphism $S: \mathbb{F} \to \mathbb{F}$ such that, for all $X \in \mathcal{C}$,
\begin{equation}
  \label{eq:F-def-antipode}
  S \circ i(\unitobj; X) = (\eval_X \otimes i(\unitobj; X^*)) \circ (X^* \otimes \sigma_{X^{**} \otimes X^*, X}) \circ (\coev_{X^*} \otimes X^{*} \otimes X).
\end{equation}
There is a unique morphism $\omega_{\mathcal{C}}: \mathbb{F} \otimes \mathbb{F} \to \unitobj$ such that
\begin{equation}
  \label{eq:F-def-omega}
  \omega_{\mathcal{C}} \circ (i(\unitobj; X) \otimes i(\unitobj; Y))
  = (\eval_X \otimes \eval_Y) \circ (X^* \otimes \sigma_{Y^*,X} \sigma_{X,Y^*} \otimes Y)
\end{equation}
for all $X, Y \in \mathcal{C}$. The morphism $\omega = \omega_{\mathcal{C}}$ is a Hopf pairing in the sense that the following equations hold:
\begin{gather}
  \label{eq:F-Hopf-pairing-1}
  \omega \circ (m \otimes \mathbb{F})
  = \omega \circ (\mathbb{F} \otimes \omega \otimes \mathbb{F}) \circ (\mathbb{F} \otimes \Delta), \\
  \label{eq:F-Hopf-pairing-2}
  \omega \circ (\mathbb{F} \otimes m)
  = \omega \circ (\mathbb{F} \otimes \omega \otimes \mathbb{F})
  \circ (\Delta \otimes \mathbb{F}), \\
  \label{eq:F-Hopf-pairing-3}
  \omega \circ (u \otimes \mathbb{F})
  = \varepsilon = \omega \circ (\mathbb{F} \otimes u).
\end{gather}
See, {\it e.g.}, \cite{MR1352517} for the detail. Now we introduce the following terminology:

\begin{definition}
  \label{def:non-dege}
  We say that the braided finite tensor category $\mathcal{C}$ is {\em non-degenerate} if the Hopf pairing $\omega = \omega_{\mathcal{C}}$ is non-degenerate in the sense that the composition
  \begin{equation}
    \label{eq:omega-sharp}
    \mathbb{F}
    \xrightarrow{\quad \id \otimes \coev \quad}
    \mathbb{F} \otimes \mathbb{F} \otimes \mathbb{F}^*
    \xrightarrow{\quad \omega \otimes \id \quad}
    \mathbb{F}^*
  \end{equation}
  is an isomorphism in $\mathcal{C}$, or, equivalently, the composition
  \begin{equation}
    \label{eq:omega-flat}
    \mathbb{F}
    \xrightarrow{\quad \coev \otimes \id \quad}
    {}^*\mathbb{F} \otimes \mathbb{F} \otimes \mathbb{F}
    \xrightarrow{\quad \id \otimes \omega \quad}
    {}^* \mathbb{F}
  \end{equation}
  is an isomorphism in $\mathcal{C}$.
\end{definition}

\subsection{Factorizability}
\label{subsec:factorizability}

Let $\mathcal{C}$ be a braided tensor category with braiding $\sigma$, and let $\overline{\mathcal{C}}$ denote the finite tensor category $\mathcal{C}^{\rev}$ equipped with the braiding $\overline{\sigma}$ given by
\begin{equation*}
  (\overline{\sigma}_{V,W}: V \otimes^{\rev} W \to W \otimes^{\rev} V)
  := (\sigma_{V,W}^{-1}: W \otimes V \to V \otimes W)
\end{equation*}
for $V, W \in \mathcal{C}$. By using the universal property of the Deligne tensor product, we define the $k$-linear exact functor $G$ by
\begin{equation}
  \label{eq:factor-functor}
  G: \mathcal{C} \boxtimes \overline{\mathcal{C}} \to \mathcal{Z}(\mathcal{C}),
  \quad V \boxtimes W \mapsto (V \otimes W, c),
\end{equation}
where the natural isomorphism $c$ is given by
\begin{equation}
  \label{eq:factor-functor-def-c}
  c_X: V \otimes W \otimes X
  \xrightarrow{\ \id_V^{} \otimes \sigma_{X,W}^{-1} \ }
  V \otimes X \otimes W
  \xrightarrow{\ \sigma_{V,X}^{} \otimes \id_W^{} \ }
  X \otimes V \otimes W
\end{equation}
for $X \in \mathcal{C}$. The functor $G$ is in fact a braided tensor functor with the monoidal structure given by the braiding of $\mathcal{Z}(\mathcal{C})$.

\begin{definition}[Etingof, Nikshych and Ostrik {\cite{MR2097289}}]
  The braided finite tensor category $\mathcal{C}$ is {\em factorizable} if the functor \eqref{eq:factor-functor} is an equivalence.
\end{definition}

Now we can state the main result of this section as follows:

\begin{theorem}
  \label{thm:factor-nondege}
  A braided finite tensor category is factorizable if and only if it is non-degenerate.
\end{theorem}

\subsection{Restriction-of-scalars functor}
\label{subsec:restriction}

To prove Theorem~\ref{thm:factor-nondege}, we first prepare a technical lemma for functors induced by algebra morphisms. Let $\mathcal{C}$ be a finite tensor category which is not necessarily braided. Given an algebra $A$ in $\mathcal{C}$, we denote by $\mathcal{C}_A$ the category of right $A$-modules in $\mathcal{C}$. A morphism $\phi: A \to B$ of algebras in $\mathcal{C}$ induces a functor
\begin{equation*}
  \mathrm{Res}_{\phi}: \mathcal{C}_B \to \mathcal{C}_A,
  \quad (M, \mu)
  \mapsto (M, \mu \circ (\id_M \otimes \phi)),
\end{equation*}
which we call the {\em restriction of scalars} along $\phi$. We remark:

\begin{lemma}
  \label{lem:res-equiv}
  $\mathrm{Res}_{\phi}$ is an equivalence if and only if $\phi$ is an isomorphism.
\end{lemma}
\begin{proof}
  This lemma is proved by generalizing the standard argument in the ordinary ring theory to the $\mathcal{C}$-enriched setting (see, {\it e.g.}, \cite{MR0450361,MR0498792,MR0498793}): Let $\iHom_B$ be the internal Hom functor for the left $\mathcal{C}$-module category $\mathcal{C}_B$. If we view $B$ as an $A$-$B$-bimodule by $\phi$, then $\mathrm{Res}_{\phi}$ is isomorphic to
  \begin{equation*}
    R = \iHom_B(B, -): \mathcal{C}_B \to \mathcal{C}_A.
  \end{equation*}
  The functor $L = (-) \otimes_A B: \mathcal{C}_A \to \mathcal{C}_B$ is left adjoint to $R$. Let $\eta$ be the unit of the adjunction $L \dashv R$. Then one can check that the composition
  \begin{equation*}
    A \xrightarrow{\quad \eta_A \quad}
    R L (A)
    \cong \iHom_B(B, A \otimes_A B)
    \cong \iHom_B(B, B) \cong B
  \end{equation*}
  coincides with $\phi$. Thus, if $\mathrm{Res}_{\phi}$ is an equivalence, then $\eta$ is an isomorphism, and hence $\phi$ is an isomorphism. The converse is clear.
\end{proof}

Let $F$ be a coalgebra in $\mathcal{C}$. Then $A = {}^* \! F$ is an algebra in $\mathcal{C}$ as the image of the algebra $F \in \mathcal{C}^{\op}$ under the tensor functor ${}^* (-): \mathcal{C}^{\op} \to \mathcal{C}^{\rev}$. If $(M, \delta)$ is a right\footnote{Since $\mathcal{C}$ is not assumed to be braided, the algebra $A$ acts on a right $F$-comodule from the right unlike the case of ordinary coalgebras over a field.} $F$-comodule in $\mathcal{C}$, then the object $M$ is a right $A$-module by
\begin{equation*}
  M \otimes A
  \xrightarrow{\quad \delta \otimes \id \quad}
  M \otimes F \otimes A
  \xrightarrow{\quad \id \otimes \eval \quad} M.
\end{equation*}
We identify the category $\mathcal{C}^F$ of right $F$-comodules with the category $\mathcal{C}_A$ of right $A$-modules by this correspondence.

Now we suppose that $\mathcal{C}$ is braided. As we have seen, the coend $\mathbb{F}$ is a Hopf algebra in $\mathcal{C}$ endowed with a canonical Hopf pairing $\omega = \omega_{\mathcal{C}}$. By~\eqref{eq:F-Hopf-pairing-1}--\eqref{eq:F-Hopf-pairing-3}, the algebra $\mathbb{F}$ acts on every right $\mathbb{F}$-comodule $(M, \delta)$ by
\begin{equation}
  \label{eq:omega-nat-action}
  a_{\delta}: M \otimes \mathbb{F}
  \xrightarrow{\quad \delta \otimes \id \quad}
  M \otimes \mathbb{F} \otimes \mathbb{F}
  \xrightarrow{\quad \id \otimes \omega \quad}
  M.
\end{equation}
Hence we get a functor
\begin{equation}
  \label{eq:omega-nat}
  \omega^{\natural}: \mathcal{C}^{\mathbb{F}} \to \mathcal{C}_{\mathbb{F}},
  \quad (M, \delta) \mapsto (M, a_{\delta}).
\end{equation}
If we identify $\mathcal{C}^{\mathbb{F}}$ with the category of right ${}^* \mathbb{F}$-modules, then the functor $\omega^{\natural}$ corresponds to the restriction of scalars along
 \eqref{eq:omega-flat}. Thus, by Lemma~\ref{lem:res-equiv}, we have:

\begin{lemma}
  $\omega^{\natural}$ is an equivalence if and only if $\mathcal{C}$ is non-degenerate.
\end{lemma}

The proof of Theorem~\ref{thm:factor-nondege} is outlined as follows: Below, we see that there are equivalences $\mathcal{C}^{\mathbb{F}} \approx \mathcal{C} \boxtimes \mathcal{C}$ and $\mathcal{C}_{\mathbb{F}} \approx \mathcal{Z}(\mathcal{C})$ of categories (Lemmas~\ref{lem:F-comodules} and~\ref{lem:F-modules}). A key observation for the proof is that the composition
\begin{equation*}
  \mathcal{C} \boxtimes \mathcal{C}
  \xrightarrow{\quad \approx \quad}
  \mathcal{C}^{\mathbb{F}}
  \xrightarrow{\quad \omega^{\natural} \quad}
  \mathcal{C}_{\mathbb{F}}
  \xrightarrow{\quad \approx \quad}
  \mathcal{Z}(\mathcal{C})
\end{equation*}
is isomorphic to the functor $G$ defined by~\eqref{eq:factor-functor}. Once this fact is recognized, it is obvious that $G$ is an equivalence if and only if $\omega^{\natural}$ is, and thus if and only if $\mathcal{C}$ is non-degenerate by the above lemma.

\subsection{Representation theory of the Hopf algebra $\mathbb{F}$}
\label{subsec:rep-th-of-F}

Let $\mathcal{C}$ be a braided finite tensor category, and let $\mathbb{F} \in \mathcal{C}$ be the Hopf algebra defined in \S\ref{subsec:non-dege}. We now give a description of the category of (co)modules over $\mathbb{F}$. For this purpose, it is convenient to use the graphical technique to express morphisms in $\mathcal{C}$. In our convention, the source and the target of a morphism are placed at the top and the bottom of the picture, respectively. If $X$ and $Y$ are objects of a braided finite tensor category with braiding $\sigma$, then the evaluation $\eval_X$, the coevaluation $\coev_X$, the braiding $\sigma_{X,Y}$, its inverse are expressed as in Figure~\ref{fig:gra-cal}. Some morphisms related to the Hopf algebra $\mathbb{F}$ are also expressed by special diagrams as in the figure.

\begin{figure}
  \input{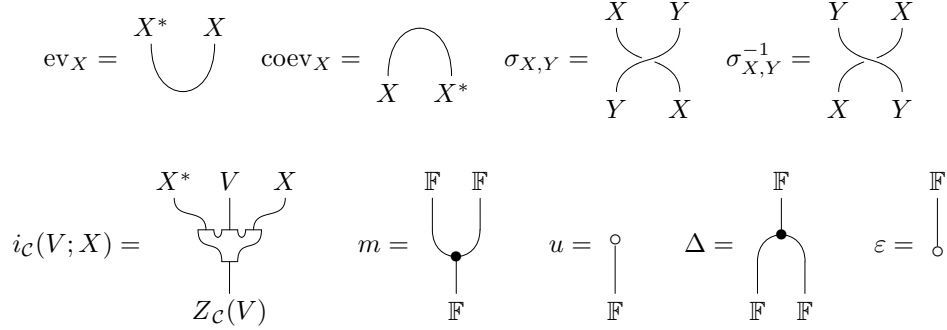}
  \caption{Graphical convention}
  \label{fig:gra-cal}
\end{figure}

For an object $V \in \mathcal{C}$, we define a morphism $\rho_V$ by
\begin{equation}
  \label{eq:can-coact-def}
  \rho_V: V
  \xrightarrow{\quad \coev \otimes \id \quad}
  V \otimes V^* \otimes V
  \xrightarrow{\quad \id \otimes i(\unitobj; V) \quad}
  V \otimes \mathbb{F}.
\end{equation}
By the definition of the comultiplication of $\mathbb{F}$, it is easy to see that the object $X$ is a right $\mathbb{F}$-comodule with coaction $\rho_V$. Thus we call $\rho_V$ the {\em canonical coaction} of $\mathbb{F}$. We now give the following description of the category of $\mathbb{F}$-comodules:

\begin{lemma}
  \label{lem:F-comodules}
  The following functor is an equivalence:
  \begin{equation*}
    \mathcal{C} \boxtimes \mathcal{C} \to \mathcal{C}^{\mathbb{F}},
    \quad V \boxtimes W
    \mapsto (V \otimes W, \id_V \otimes \rho_W).
  \end{equation*}
\end{lemma}

This is a special case of Lyubashenko's result \cite[\S2.7]{MR1625495}. The original proof uses the notion of squared coalgebras. There is also a proof based on the theory of module categories over finite tensor categories \cite[Lemma 3.5]{2016arXiv160805905S}.

The category of $\mathbb{F}$-modules is described as follows:

\begin{lemma}
  \label{lem:F-modules}
  There is an isomorphism $\mathcal{C}_{\mathbb{F}} \cong \mathcal{Z}(\mathcal{C})$ of categories.
\end{lemma}

This is a special case of Majid's result \cite[Theorem 3.2]{MR1192156}. We give a different (but essentially same) proof by enlightening the role of the central Hopf monad.

\begin{proof}
  For $V \in \mathcal{C}$, we define $\xi_V: V \otimes \mathbb{F} \to Z_{\mathcal{C}}(V)$ by
  \begin{equation}
    \label{eq:monad-iso-1}
    \xi_V \circ (\id_V \otimes i_{\mathcal{C}}(\unitobj; X)) = i_{\mathcal{C}}(V; X) \circ (\sigma_{V,X^*} \otimes \id_X)
  \end{equation}
  for all $X \in \mathcal{C}$. Then we have
  \begin{equation}
    \label{eq:monad-iso-2}
    \eta_V = \xi_V \circ (\id_V \otimes u)
    \quad \text{and} \quad
    \mu_V \circ \xi_{Z_{\mathcal{C}}(V)} \circ (\xi_{V} \otimes \mathbb{F})
    = \xi_V \circ (\id_V \otimes m)
  \end{equation}
  for all $V \in \mathcal{C}$. Indeed, it is easy to check the first equation. The second one can be verified as in Figure~\ref{fig:proof-F-mod-Z}. Equation~\eqref{eq:monad-iso-2} means that $\xi$ is in fact an isomorphism of monads. Thus their categories of modules are also isomorphic.
\end{proof}

Let $(V, c)$ be an object of $\mathcal{Z}(\mathcal{C})$. By the construction of the category isomorphism $\mathcal{Z}(\mathcal{C}) \cong \mathcal{C}_{\mathbb{F}}$ in the above lemma, the right $\mathbb{F}$-module corresponding to $(V, c)$ is the object $V$ with the action
\begin{equation*}
  \triangleleft_V: V \otimes \mathbb{F}
  \xrightarrow{\quad \xi_V \quad}
  Z_{\mathcal{C}}(V)
  \xrightarrow{\quad a \quad} V,
\end{equation*}
where $a$ is the action of $Z_{\mathcal{C}}$ on $V$ defined by~\eqref{eq:center-to-Z-mod}. Thus,
\begin{equation}
  \label{eq:Z-mod-to-F-mod}
  \begin{aligned}
    \triangleleft_V \circ (\id_V \otimes i_{\mathcal{C}}(\unitobj; X))
    & = (\eval_X \otimes \id_V) \circ (\id_{X^*} \otimes c_X) \circ (\sigma_{V,X^*} \otimes \id_X).
  \end{aligned}
\end{equation}

\begin{figure}
  \input{fig-S3-2}
  \caption{Proof of Equation~\eqref{eq:monad-iso-2}}
  \label{fig:proof-F-mod-Z}

  \bigskip
  \input{fig-S3-3}
  \caption{Proof of $\triangleleft_1 = \triangleleft_2$}
  \label{fig:F-actions}
\end{figure}

\subsection{Proof of Theorem~\ref{thm:factor-nondege}}

We consider the following diagram:
\begin{equation*}
  \xymatrix{
    \mathcal{C} \boxtimes \overline{\mathcal{C}}
    \ar[rrr]^{\approx}_{\text{\rm Lemma~\ref{lem:F-comodules}}}
    \ar[d]_{G}
    & & & \mathcal{C}^{\mathbb{F}}
    \ar[d]^{\omega^{\natural}} \\
    \mathcal{Z}(\mathcal{C})
    \ar[rrr]^{\approx}_{\text{\rm Lemma~\ref{lem:F-modules}}}
    & & & \mathcal{C}_{\mathbb{F}}
  }
\end{equation*}
Let $F_1: \mathcal{C} \boxtimes \overline{\mathcal{C}} \to \mathcal{C}^{\mathbb{F}}$ and $F_2: \mathcal{Z}(\mathcal{C}) \to \mathcal{C}^{\mathbb{F}}$ be the equivalences given in Lemmas~\ref{lem:F-comodules} and \ref{lem:F-modules}, respectively, and set $E_1 = \omega^{\natural} F_1$ and $E_2 = F_2 G$. For all objects $V, W \in \mathcal{C}$, the underlying object of the $\mathbb{F}$-module $E_i(V \boxtimes W)$ is $V \otimes W$. Let $\triangleleft_i$ ($i = 1, 2$) denote the action of $\mathbb{F}$ on $E_i(V \boxtimes W)$. Figure~\ref{fig:F-actions} shows $\triangleleft_1 = \triangleleft_2$. Thus we have
\begin{equation*}
  E_1(V \boxtimes W) = E_2(V \boxtimes W).
\end{equation*}
This implies that $E_1 \cong E_2$ as functors from $\mathcal{C} \boxtimes \overline{\mathcal{C}}$ to $\mathcal{C}_{\mathbb{F}}$, that is, the diagram in concern commutes up to isomorphisms. By the argument at the last of \S\ref{subsec:restriction}, we conclude that $\mathcal{C}$ is factorizable if and only if $\mathcal{C}$ is non-degenerate. The proof is done.

\section{Factorizability and the M\"uger center}
\label{sec:factor-vs-tricen}

\subsection{The M\"uger center}

Let $\mathcal{C}$ be a braided finite tensor category with braiding $\sigma$, and let $\mathcal{D}$ be a full subcategory of $\mathcal{C}$. The {\em M\"uger centralizer} of $\mathcal{D}$ in $\mathcal{C}$ is the full subcategory of $\mathcal{C}$ consisting of all objects $V \in \mathcal{C}$ such that $\sigma_{V,X} \circ \sigma_{X,V} = \id_{X \otimes V}$ for all $X \in \mathcal{D}$. The M\"uger centralizer of $\mathcal{D}$ in $\mathcal{C}$ is denoted by $\Mueg_{\mathcal{C}}(\mathcal{D})$, or simply by $\mathcal{D}'$ if the ambient category $\mathcal{C}$ is clear from the context.

\begin{definition}
  We call $\mathcal{C}' = \Mueg_{\mathcal{C}}(\mathcal{C})$ the {\em M\"uger center} of $\mathcal{C}$. We say that {\em the M\"uger center of $\mathcal{C}$ is trivial} if every object of $\mathcal{C}'$ is isomorphic to the direct sum of finitely many copies of the unit object $\unitobj \in \mathcal{C}$, or, equivalently, $\mathcal{C}' \approx \mathrm{Vec}$.
\end{definition}

It has been known that a ribbon fusion category is a modular tensor category if and only if its M\"uger center is trivial \cite{MR1741269,MR1966525,MR3242743}. The main purpose of this section is to generalize this fact to the non-semisimple case as follows:

\begin{theorem}
  \label{thm:factor-tricen}
  A braided finite tensor category is factorizable if and only if its M\"uger center is trivial.
\end{theorem}

\subsection{Dimensions of tensor full subcategories}
\label{subsec:FPdim-subsec}

Theorem~\ref{thm:factor-tricen} will be proved by using a formula of the Frobenius-Perron dimension of the tensor full subcategory generated by two tensor full subcategories (Lemma~\ref{lem:FPdim-cap}). We first clarify what we mean by a tensor full subcategory:

\begin{definition}
  \label{def:ten-ful-subcat}
  A {\em topologizing subcategory} \cite[\S5.3]{MR1347919} of an abelian category $\mathcal{A}$ is a full subcategory of $\mathcal{A}$ closed under finite direct sums and subquotients. Let $\mathcal{C}$ be a finite tensor category. By a {\em tensor full subcategory} of $\mathcal{C}$, we mean a topologizing subcategory of $\mathcal{C}$ closed under the tensor product and the duality functors.
\end{definition}

We remark:

\begin{lemma}
  \label{lem:topologizing}
  Let $\mathcal{S}$ be a topologizing subcategory of a finite abelian category $\mathcal{A}$. Then $\mathcal{S}$ itself is a finite abelian category such that the inclusion functor $i: \mathcal{S} \hookrightarrow \mathcal{A}$ preserves and reflects exact sequences.
\end{lemma}
\begin{proof}
  Let $M \in \mathcal{A}$ be an object. If $X$ and $Y$ are subobjects of $M$ belonging to $\mathcal{S}$, then their sum $X + Y \subset M$ also belongs to $\mathcal{S}$ as a quotient of $X \oplus Y \in \mathcal{S}$. From this observation, we see that $M$ has the largest subobject, say $t(M) \subset M$, belonging to $\mathcal{S}$. Since $\mathcal{S}$ is closed under quotient objects, the assignment $M \mapsto t(M)$ extends to a $k$-linear functor $t: \mathcal{A} \to \mathcal{S}$. Moreover, we have
  \begin{equation*}
    \Hom_{\mathcal{A}}(V, M) = \Hom_{\mathcal{S}}(V, t(M))
  \end{equation*}
  for all $V \in \mathcal{S}$ and $M \in \mathcal{A}$. Namely, the functor $t$ is a right adjoint of $i$.

  Applying the same argument to $i^{\op}: \mathcal{S}^{\op} \to \mathcal{A}^{\op}$, we see that $i^{\op}$ has a right adjoint (the condition that $\mathcal{S}$ is closed under subobjects is equivalent to that $\mathcal{S}^{\op}$ is closed under quotient objects). Hence, $i$ also has a left adjoint.

  Now we consider the $k$-linear idempotent comonad $T := i \circ t$ on $\mathcal{A}$ associated to the adjunction $i \dashv t$. Since $i$ and $t$ has left adjoints, so does $T$. The full subcategory $\mathcal{S}$ can be identified with the category $\mathcal{A}^T$ of $T$-comodules and, under this identification, the forgetful functor $\mathcal{A}^T \to \mathcal{A}$ corresponds to the inclusion functor $\mathcal{S} \hookrightarrow \mathcal{A}$. The claim of this lemma now follows from Lemma~\ref{lem:finiteness-lemma}.
\end{proof}

Now let $\mathcal{C}$ be a finite tensor category. For a topologizing subcategory $\mathcal{S}$ of $\mathcal{C}$, we denote by $t_{\mathcal{S}}: \mathcal{C} \to \mathcal{S}$ the functor defined by taking the largest subobject belonging to $\mathcal{S}$. By the proof of the above lemma, $t_{\mathcal{S}}$ is $k$-linear and left exact. We often regard the functor $t_{\mathcal{S}}$ as a $k$-linear left exact endofunctor on $\mathcal{C}$ by composing the inclusion functor. By Lemmas~\ref{lem:coend-adj} and~\ref{lem:coend-exist}, the coends
\begin{equation*}
  Z_{\mathcal{S}}(V) := \int^{X \in \mathcal{S}} \!\!\! X^* \otimes V \otimes X
  \quad \text{and} \quad \int^{X \in \mathcal{C}} \!\!\! X^* \otimes V \otimes t_{\mathcal{S}}(X)
\end{equation*}
exist for each $V \in \mathcal{C}$ and are canonically isomorphic. The assignment $V \mapsto Z_{\mathcal{S}}(V)$ gives rise to a $k$-linear endofunctor on $\mathcal{C}$.

We are interested in the case where $\mathcal{S}$ is a tensor full subcategory of $\mathcal{C}$. If this is the case, then one can endow $Z_{\mathcal{S}}$ with the structure of a Hopf monad in a similar way as the central Hopf monad. By the same argument as in \S\ref{subsec:Hopf-monad-Z}, we have:

\begin{lemma}
  \label{lem:Hopf-monad-ZS}
  The monoidal category of $Z_{\mathcal{S}}$-modules can be identified with the centralizer $\mathcal{Z}(\mathcal{S}; \mathcal{C})$ introduced in Subsection~\ref{subsec:dri-cen}.
\end{lemma}

An object $(V, c) \in \mathcal{Z}(\mathcal{C})$ becomes an object of $\mathcal{Z}(\mathcal{S}; \mathcal{C})$ by restricting the natural isomorphism $c = \{ c_X \}_{X \in \mathcal{C}}$ to $X \in \mathcal{S}$. The forgetful functor from $\mathcal{Z}(\mathcal{C})$ to  $\mathcal{Z}(\mathcal{S}; \mathcal{C})$ is defined in this manner. We remark:

\begin{lemma}
  \label{lem:FPdim-center-wrt-S}
  $\mathcal{Z}(\mathcal{S}; \mathcal{C})$ is a finite tensor category such that
  \begin{equation}
    \label{eq:FPdim-center-wrt-S}
    \FPdim(\mathcal{Z}(\mathcal{S}; \mathcal{C}))
    = \FPdim(\mathcal{S}) \FPdim(\mathcal{C}),
  \end{equation}
  and the forgetful functor $U: \mathcal{Z}(\mathcal{C}) \to \mathcal{Z}(\mathcal{S}; \mathcal{C})$ is dominant.
\end{lemma}
\begin{proof}
  We use the theory of exact module categories and the dual tensor category; see \cite[Chapter 7]{MR3242743}. We consider the finite tensor categories $\mathcal{D} := \mathcal{S} \boxtimes \mathcal{C}^{\rev}$ and $\mathcal{E} := \mathcal{C} \boxtimes \mathcal{C}^{\rev}$. The category $\mathcal{C}$ is an exact $\mathcal{E}$-module category by
  \begin{equation*}
    (X \boxtimes Y) \triangleright V = X \otimes V \otimes Y
    \quad (X, Y, M \in \mathcal{C}).
  \end{equation*}
  Let $i: \mathcal{D} \to \mathcal{E}$ be the tensor functor induced by the inclusion functor $\mathcal{S} \hookrightarrow \mathcal{C}$.
  Then $\mathcal{D}$ acts on $\mathcal{C}$ through $i$. Since the base field is assumed to be algebraically closed, every indecomposable projective object of $\mathcal{D}$ is of the form $P \boxtimes Q$ for some indecomposable projective objects $P \in \mathcal{S}$ and $Q \in \mathcal{C}$. Although $P$ is not projective in $\mathcal{C}$ in general, the object $(P \boxtimes Q) \triangleright M = P \otimes M \otimes Q$ is projective in $\mathcal{C}$ for all $M \in \mathcal{C}$ since $Q$ is projective in $\mathcal{C}$. Thus $\mathcal{C}$ is an exact $\mathcal{D}$-module category.

  It is known that $\mathcal{Z}(\mathcal{C})$ is equivalent to the dual tensor category of $\mathcal{E}$ with respect to $\mathcal{C}$. In the same way, we see that $\mathcal{Z}(\mathcal{S}; \mathcal{C})$ is equivalent to the dual tensor category of $\mathcal{D}$ with respect to $\mathcal{C}$. This implies that $\mathcal{Z}(\mathcal{S}; \mathcal{C})$ is a finite tensor category. Now, by the Morita invariance of the Frobenius-Perron dimension, we have
  \begin{equation*}
    \FPdim \mathcal{Z}(\mathcal{S}; \mathcal{C})
    = \FPdim(\mathcal{D})
    = \FPdim(\mathcal{S}) \FPdim(\mathcal{C}).
  \end{equation*}
  The functor $U$ is dual to $(i, \mathcal{C})$ in the sense of \cite{MR3242743}. Since $i$ is fully faithful, $U$ is dominant by \cite[Theorem 7.17.4]{MR3242743}.
\end{proof}

We now compute the Frobenius-Perron dimensions of $\mathbb{F}_{\mathcal{S}} := Z_{\mathcal{S}}(\unitobj)$ and
\begin{equation*}
  \widetilde{\mathbb{F}}_{\mathcal{S}} := \int^{X \in \mathcal{S}} \!\!\! X \boxtimes X \in \mathcal{C}^{\op} \boxtimes \mathcal{C}.
\end{equation*}

\begin{lemma}
  \label{lem:FPdim-coend}
  $\FPdim(\widetilde{\mathbb{F}}_{\mathcal{S}}) = \FPdim(\mathbb{F}_{\mathcal{S}}) = \FPdim(\mathcal{S})$.
\end{lemma}
\begin{proof}
  We define $F: \mathcal{C}^{\op} \boxtimes \mathcal{C} \to \mathcal{C}$ by $F(X \boxtimes Y) = X^* \otimes Y$ ($X, Y \in \mathcal{C}$).
  By definition, $F$ is $k$-linear and exact. For all objects $X, Y \in \mathcal{C}$, we have
  \begin{equation*}
    \FPdim F(X \boxtimes Y)
    = \FPdim(X) \FPdim(Y)
    = \FPdim(X \boxtimes Y).
  \end{equation*}
  Every simple object of $\mathcal{C}^{\op} \boxtimes \mathcal{C}$ is of the form $X \boxtimes Y$ for some simple objects $X$ and $Y$ of $\mathcal{C}$. Since $F$ is exact, and since the Frobenius-Perron dimension of an object depends only on the composition factors of the object, we have
  \begin{equation*}
    \FPdim F(M) = \FPdim(M)
  \end{equation*}
  for all $M \in \mathcal{C}^{\op} \boxtimes \mathcal{C}$. The first equality is the case where $M = \widetilde{\mathbb{F}}_{\mathcal{S}}$.

  To prove the second equality, we consider the Hopf monad $Z_{\mathcal{S}}$ introduced in the above. If we identify $\mathcal{Z}(\mathcal{S}; \mathcal{C})$ with the category of $Z_{\mathcal{S}}$-modules by Lemma~\ref{lem:Hopf-monad-ZS}, then the free $Z_{\mathcal{S}}$-module functor
  \begin{equation*}
    L: \mathcal{C} \to \mathcal{Z}(\mathcal{S}; \mathcal{C}),
    \quad V \mapsto Z_{\mathcal{S}}(V)
  \end{equation*}
  is left adjoint to the forgetful functor $\mathcal{Z}(\mathcal{S}; \mathcal{C}) \to \mathcal{C}$. By~\eqref{eq:FPdim-ind}, \eqref{eq:FPdim-center-wrt-S} and the fact that tensor functors preserve Frobenius-Perron dimensions, we have
  \begin{equation*}
    \FPdim(\mathbb{F}_{\mathcal{S}})
    = \FPdim(L(\unitobj))
    = \frac{\FPdim(\mathcal{Z}(\mathcal{S}; \mathcal{C}))}{\FPdim(\mathcal{C})} \FPdim(\unitobj)
    = \FPdim(\mathcal{S}). \qedhere
  \end{equation*}
\end{proof}

For a $k$-linear functor $F: \mathcal{M} \to \mathcal{N}$ between finite abelian categories, we denote by $\mathrm{Im}(F)$ the full subcategory of $\mathcal{N}$ consisting of all subquotients of objects of the form $F(X)$ for some $X \in \mathcal{M}$. Let $\mathcal{A}$ and $\mathcal{B}$ be topologizing subcategories of $\mathcal{C}$, and let $T: \mathcal{A} \boxtimes \mathcal{B} \to \mathcal{C}$ be the functor induced by the tensor product of $\mathcal{C}$. We set $\mathcal{A} \vee \mathcal{B} = \mathrm{Im}(T)$. It is easy to see that $\mathcal{A} \cap \mathcal{B}$ and $\mathcal{A} \vee \mathcal{B}$ are topologizing full subcategories of $\mathcal{C}$ as well.

Now we suppose that $\mathcal{A}$ and $\mathcal{B}$ are tensor full subcategories of $\mathcal{C}$ and, moreover, the functor $T: \mathcal{A} \boxtimes \mathcal{B} \to \mathcal{C}$ has a structure of a tensor functor (this assumption is satisfied if, for example, $\mathcal{C}$ is braided). Then $\mathcal{A} \vee \mathcal{B}$ is also a tensor full subcategory of $\mathcal{C}$. We give the following formula for the Frobenius-Perron dimension of $\mathcal{A} \vee \mathcal{B}$ (see \cite[Lemma 8.21.6]{MR3242743} for the semisimple case):

\begin{lemma}
  \label{lem:FPdim-cap}
  Under the above assumptions, we have
  \begin{equation*}
    \FPdim(\mathcal{A} \vee \mathcal{B}) \FPdim(\mathcal{A} \cap \mathcal{B})
    = \FPdim(\mathcal{A}) \FPdim(\mathcal{B}).
  \end{equation*}
\end{lemma}
\begin{proof}
  We consider the following tensor functor:
  \begin{equation*}
    F: \mathcal{A}^{\op,\rev} \boxtimes \mathcal{B} \to \mathcal{A} \vee \mathcal{B},
    \quad X \boxtimes Y \mapsto T(X^* \boxtimes Y) = X^* \otimes Y.
  \end{equation*}
  Since $F$ is exact, it has a right adjoint, say $R$. By \eqref{eq:FPdim-ind}, we have
  \begin{equation*}
    \FPdim(\mathcal{A} \vee \mathcal{B}) \FPdim(R(\unitobj))
    = \FPdim(\mathcal{A}^{\op,\rev} \boxtimes \mathcal{B})
    = \FPdim(\mathcal{A}) \FPdim(\mathcal{B}).
  \end{equation*}
  Thus, to prove this lemma, it suffices to show that
  \begin{equation}
    \label{eq:FPdim-cap-pf-1}
    \FPdim(R(\unitobj)) = \FPdim(\mathcal{A} \cap \mathcal{B}).
  \end{equation}
  Let $\Phi = \Phi_{\mathcal{A}, \mathcal{B}}$ and $\overline{\Phi} = \overline{\Phi}_{\mathcal{A}, \mathcal{B}}$ be the category equivalences given by~\eqref{eq:def-Phi} and \eqref{eq:def-Phi-inverse}, respectively. Set $\mathcal{L} = \LEX(\mathcal{A}, \mathcal{B})$. For every $E \in \mathcal{L}$, we compute:
  \begin{align*}
    \Hom_{\mathcal{L}}(E, \Phi(R(\unitobj)))
    & \cong \Hom_{\mathcal{A}^{\op} \boxtimes \mathcal{B}}(\overline{\Phi}(E), R(\unitobj)) \\
    & \cong \textstyle \int_{X \in \mathcal{A}} \Hom_{\mathcal{A}^{\op} \boxtimes \mathcal{B}}(X \boxtimes E(X), R(\unitobj)) \\
    & \cong \textstyle \int_{X \in \mathcal{A}} \Hom_{\mathcal{C}}(X^* \otimes E(X), \unitobj) \\
    & \cong \textstyle \int_{X \in \mathcal{A}} \Hom_{\mathcal{C}}(E(X), X) \\
    & \cong \textstyle \int_{X \in \mathcal{A}} \Hom_{\mathcal{B}}(E(X), t_{\mathcal{A} \cap \mathcal{B}} (X)) \\
    & \cong \Hom_{\mathcal{L}}(E, t_{\mathcal{A} \cap \mathcal{B}}),
  \end{align*}
  where $t_{\mathcal{A} \cap \mathcal{B}}: \mathcal{C} \to \mathcal{A} \cap \mathcal{B}$ is regarded as an object of $\mathcal{L}$ by composing the inclusion functors. By the Yoneda lemma and Lemma~\ref{lem:coend-adj}, we have
  \begin{equation*}
    R(\unitobj) \cong \overline{\Phi}(t_{\mathcal{A} \cap \mathcal{B}})
    \cong \int^{X \in \mathcal{A}} X \boxtimes t_{\mathcal{A} \cap \mathcal{B}} (X)
    \cong \int^{X \in \mathcal{A} \cap \mathcal{B}} X \boxtimes X
    = \widetilde{\mathbb{F}}_{\mathcal{A} \cap \mathcal{B}}.
  \end{equation*}
  Now \eqref{eq:FPdim-cap-pf-1} follows from Lemma~\ref{lem:FPdim-coend}. The proof is done.
\end{proof}

\subsection{Proof of Theorem~\ref{thm:factor-tricen}}

Let $\mathcal{C}$ be a braided finite tensor category with braiding $\sigma$. We prove the main theorem of this section which states that $\mathcal{C}$ is factorizable if and only if the M\"uger center of $\mathcal{C}$ is trivial.

We define tensor full subcategories $\mathcal{C}_+$ and $\mathcal{C}_-$ of $\mathcal{Z}(\mathcal{C})$ by $\mathcal{C}_{\pm} = \mathrm{Im}(T_{\pm})$, where $T_+$ and $T_-$ are braided tensor functors defined by
\begin{equation}
  \label{eq:C-embeds-ZC}
  T_+: \mathcal{C} \to \mathcal{Z}(\mathcal{C}), \quad V \mapsto (V, \sigma_{V,-})
  \quad \text{and} \quad
  T_-: \overline{\mathcal{C}} \to \mathcal{Z}(\mathcal{C}), \quad V \mapsto (V, \sigma_{-,V}^{-1}),
\end{equation}
respectively. The functors $T_+$ and $T_-$ are fully faithful, and therefore we have
\begin{equation*}
  \FPdim(\mathcal{C}_{\pm}) = \FPdim(\mathcal{C})
\end{equation*}
by \cite[Corollary 6.3.5]{MR3242743}. The tensor full subcategories $\mathcal{A} = \mathcal{C}_+$ and $\mathcal{B} = \mathcal{C}_-$ of $\mathcal{Z}(\mathcal{C})$ satisfy the assumptions of Lemma~\ref{lem:FPdim-cap}. Thus we have
\begin{equation}
  \label{eq:FPdim-Im-G}
  \FPdim(\mathcal{C}_+ \vee \mathcal{C}_-) \FPdim(\mathcal{C}_+ \cap \mathcal{C}_-)
  = \FPdim(\mathcal{C})^2
\end{equation}
by that lemma.

The full subcategory $\mathcal{C}_+ \vee \mathcal{C}_-$ is the image of the functor $G: \mathcal{C} \boxtimes \overline{\mathcal{C}} \to \mathcal{Z}(\mathcal{C})$ used to define the factorizability. By~\eqref{eq:FPdim-Im-G} and the basic properties of the Frobenius-Perron dimensions \cite[Section 6]{MR3242743}, we have the following logical equivalences:
\begin{align*}
  \text{$G$ is an equivalence}
  & \iff \FPdim(\mathcal{C}_+ \vee \mathcal{C}_-) = \FPdim \mathcal{Z}(\mathcal{C}) \\
  & \iff \FPdim(\mathcal{C}_+ \cap \mathcal{C}_-) = 1 \\
  & \iff \mathcal{C}_+ \cap \mathcal{C}_{-} \approx \mathrm{Vec}.
\end{align*}
By the definition of $\mathcal{C}_{+}$ and $\mathcal{C}_{-}$, the full subcategory $\mathcal{C}_+ \cap \mathcal{C}_-$ can be identified with the M\"uger center of $\mathcal{C}$. Thus $\mathcal{C}$ is factorizable ({\it i.e.}, $G$ is an equivalence) if and only if $\mathcal{C}' \approx \mathrm{Vec}$. The proof is done.

\subsection{Dimensions of the M\"uger centralizer}

Let $\mathcal{C}$ be a braided finite tensor category, and let $\mathcal{S}$ be a full subcategory. By the definition of a braiding, the full subcategory $\mathcal{S}'$ is in fact a tensor full subcategory of $\mathcal{C}$. Thus we can talk about the Frobenius-Perron dimension of $\mathcal{S}'$. We have used Lemma~\ref{lem:FPdim-cap} to prove the main theorem of this section. As an application of this lemma, we also give the following theorem:

\begin{theorem}
  \label{thm:dim-muger-cent}
  For every tensor full subcategory $\mathcal{D}$ of $\mathcal{C}$, we have
  \begin{gather}
    \label{eq:FPdim-centralizer-1}
    \FPdim(\mathcal{D}) \FPdim(\mathcal{D}') = \FPdim(\mathcal{C}) \FPdim(\mathcal{D} \cap \mathcal{C}'), \\
    \label{eq:FPdim-centralizer-2}
    \mathcal{D}'' = \mathcal{D} \vee \mathcal{C}'.
  \end{gather}
\end{theorem}

In particular, if $\mathcal{C}$ is factorizable (or, equivalently, non-degenerate), then
\begin{gather}
  \label{eq:FPdim-centralizer-1-nd}
  \FPdim(\mathcal{D}) \FPdim(\mathcal{D}') = \FPdim(\mathcal{C}), \\
  \label{eq:FPdim-centralizer-2-nd}
  \mathcal{D}'' = \mathcal{D}
\end{gather}
by Theorem~\ref{thm:factor-tricen}. Equations \eqref{eq:FPdim-centralizer-1}--\eqref{eq:FPdim-centralizer-2-nd} have been known in the semisimple case \cite{MR3242743}. Our proof goes along almost the same way as \cite[\S8.21]{MR3242743}.

\begin{proof}
  {\bf Factorizable case}.
  We first prove this theorem under the assumption that $\mathcal{C}$ is factorizable. Let $U: \mathcal{Z}(\mathcal{C}) \to \mathcal{Z}(\mathcal{D}; \mathcal{C})$ be the forgetful functor. We define the tensor full subcategories $\mathcal{C}_{+}$ and $\mathcal{C}_{-}$ of $\mathcal{Z}(\mathcal{D}; \mathcal{C})$ by
  \begin{equation*}
    \mathcal{C}_{\pm} = \mathrm{Im}(U \circ T_{\pm}),
  \end{equation*}
  where $T_{+}$ and $T_{-}$ are the functors given by~\eqref{eq:C-embeds-ZC}. By the same argument as in the proof of Theorem~\ref{thm:factor-tricen}, we have
  \begin{equation}
    \label{eq:FPdim-centra-pf-1}
    \FPdim(\mathcal{C}_{+} \cap \mathcal{C}_{-}) \cdot \FPdim(\mathcal{C}_{+} \vee \mathcal{C}_{-})
    = \FPdim(\mathcal{C})^2.
  \end{equation}
  Let $G$ be the functor used to define the factorizability. By the assumption, $G$ is an equivalence. Since $U$ is dominant as proved in Lemma~\ref{lem:FPdim-center-wrt-S}, we have
  \begin{equation*}
    \mathcal{C}_{+} \vee \mathcal{C}_{-} = \mathrm{Im}(U \circ G) = \mathrm{Im}(U) = \mathcal{Z}(\mathcal{D}; \mathcal{C}).
  \end{equation*}
  We also have $\mathcal{C}_{+} \cap \mathcal{C}_{-} \approx \mathcal{D}'$. Hence, by~\eqref{eq:FPdim-center-wrt-S} and~\eqref{eq:FPdim-centra-pf-1},
  \begin{equation*}
    \FPdim(\mathcal{D}') \cdot \FPdim(\mathcal{D}) \FPdim(\mathcal{C}) = \FPdim(\mathcal{C})^2.
  \end{equation*}
  Since $\FPdim(\mathcal{C}) \ne 0$, we obtain~\eqref{eq:FPdim-centralizer-1-nd}, which is equivalent to~\eqref{eq:FPdim-centralizer-1} in this case. By replacing $\mathcal{D}$ with $\mathcal{D}'$ in \eqref{eq:FPdim-centralizer-1}, we also obtain
  \begin{equation*}
    \FPdim(\mathcal{D}') \FPdim(\mathcal{D}'') = \FPdim(\mathcal{C}).
  \end{equation*}
  Thus $\FPdim(\mathcal{D}) = \FPdim(\mathcal{D}'')$. Since $\mathcal{D} \subset \mathcal{D}''$, we have $\mathcal{D} = \mathcal{D}''$.

  \medskip \noindent
  {\bf General case}. Now we prove the general case.
  For a full subcategory $\mathcal{S}$ of $\mathcal{Z}(\mathcal{C})$, we set $\mathcal{S}^{\checkmark} = \Mueg_{\mathcal{Z}(\mathcal{C})}(\mathcal{S})$. We regard $\mathcal{C}$ and $\mathcal{D} \subset \mathcal{C}$ as full subcategories of $\mathcal{Z}(\mathcal{C})$ by the functor $T_{+}: \mathcal{C} \to \mathcal{Z}(\mathcal{C})$ given by~\eqref{eq:C-embeds-ZC}. Since $\mathcal{Z}(\mathcal{C})$ is factorizable \cite{MR2097289}, and since this theorem has been proved in the factorizable case, we have
  \begin{equation}
    \label{eq:FPdim-C-check}
    \FPdim(\mathcal{C}^{\checkmark})
    = \frac{\FPdim(\mathcal{Z}(\mathcal{C}))}{\FPdim(\mathcal{C})}
    = \FPdim(\mathcal{C}).
  \end{equation}
  For a full subcategory $\mathcal{S} \subset \mathcal{C}$, we still denote by $\mathcal{S}' = \Mueg_{\mathcal{C}}(\mathcal{S})$ the M\"uger centralizer of $\mathcal{S}$ in $\mathcal{C}$. We also have the following equations:
  \begin{equation}
    \label{eq:FPdim-centra-pf-7}
    \mathcal{D} \cap \mathcal{C}^{\checkmark}
    = \mathcal{D} \cap \mathcal{C}'
    \quad \text{and} \quad
    (\mathcal{D} \vee \mathcal{C}^{\checkmark})^{\checkmark}
    = \mathcal{D}^{\checkmark} \cap \mathcal{C}^{\checkmark\checkmark}
    = \mathcal{D}^{\checkmark} \cap \mathcal{C}
    = \mathcal{D}'.
  \end{equation}
  Now we compute the Frobenius-Perron dimension of $\mathcal{D} \vee \mathcal{C}^{\checkmark} \subset \mathcal{Z}(\mathcal{C})$ in two different ways: First, by Lemma~\ref{lem:FPdim-cap} and equations~\eqref{eq:FPdim-centralizer-1-nd} and~\eqref{eq:FPdim-centra-pf-7}, we have
  \begin{equation}
    \label{eq:FPdim-centra-pf-2}
    \FPdim(\mathcal{D} \vee \mathcal{C}^{\checkmark})
    = \frac{\FPdim(\mathcal{D}) \FPdim(\mathcal{C}^{\checkmark})}{\FPdim(\mathcal{D} \cap \mathcal{C}^{\checkmark})}
    = \frac{\FPdim(\mathcal{D}) \FPdim(\mathcal{C})}{\FPdim(\mathcal{D} \cap \mathcal{C}')}.
  \end{equation}
  Second, by equations \eqref{eq:FPdim-C-check} and~\eqref{eq:FPdim-centra-pf-7}, we have
  \begin{equation}
    \label{eq:FPdim-centra-pf-3}
      \FPdim(\mathcal{D} \vee \mathcal{C}^{\checkmark})
      = \frac{\FPdim \mathcal{Z}(\mathcal{C})}{\FPdim((\mathcal{D} \vee \mathcal{C}^{\checkmark})^{\checkmark})}
      = \frac{\FPdim (\mathcal{C})^2}{\FPdim(\mathcal{D}')}.      
  \end{equation}
  Comparing~\eqref{eq:FPdim-centra-pf-2} and~\eqref{eq:FPdim-centra-pf-3}, we get~\eqref{eq:FPdim-centralizer-1}. Finally, we prove~\eqref{eq:FPdim-centralizer-2} as follows: By replacing $\mathcal{D}$ with $\mathcal{D}'$ in \eqref{eq:FPdim-centralizer-1}, we have
  \begin{equation}
    \label{eq:FPdim-centra-pf-4}
    \begin{aligned}
      \FPdim(\mathcal{D}') \FPdim(\mathcal{D}'')
      & = \FPdim(\mathcal{C}) \FPdim(\mathcal{D}' \cap \mathcal{C}') \\
      & = \FPdim(\mathcal{C}) \FPdim(\mathcal{C}').
    \end{aligned}
  \end{equation}
  Now we consider the tensor full subcategory $\mathcal{E} = \mathcal{C}' \vee \mathcal{D}$ of $\mathcal{C}$. The M\"uger centralizer of $\mathcal{E}$ in $\mathcal{C}$ is computed as follows:
  \begin{equation}
    \label{eq:FPdim-centra-pf-5}
    \mathcal{E}' = (\mathcal{C}' \vee \mathcal{D})' = \mathcal{C}'' \cap \mathcal{D}' = \mathcal{C } \cap \mathcal{D}' = \mathcal{D}'.
  \end{equation}
  By applying \eqref{eq:FPdim-centralizer-1} to $\mathcal{E}$, we have
  \begin{equation}
    \label{eq:FPdim-centra-pf-6}
    \begin{aligned}
      \FPdim(\mathcal{E}) \FPdim(\mathcal{E}')
      & = \FPdim(\mathcal{C}) \FPdim(\mathcal{E}' \cap \mathcal{C}') \\
      & = \FPdim(\mathcal{C}) \FPdim(\mathcal{C}').
    \end{aligned}
  \end{equation}
  By~\eqref{eq:FPdim-centra-pf-4}--\eqref{eq:FPdim-centra-pf-6}, we have $\FPdim(\mathcal{D}'') = \FPdim(\mathcal{E})$. Since $\mathcal{E} \subset \mathcal{D}''$, we conclude that $\mathcal{E} = \mathcal{D}''$. Thus~\eqref{eq:FPdim-centralizer-2} is proved.
\end{proof}

\section{Weak-factorizability}
\label{sec:weak-factori}

\subsection{Comparison to the original definition}
\label{subsec:vs-S-matrix}

In this section, we first explain the relation between the non-degeneracy of a braided finite tensor category and the $S$-matrix of a ribbon fusion category. Then we give another condition for braided finite tensor categories that is equivalent to the non-degeneracy.

Let $\mathcal{C}$ be a finite tensor category. There are two vector spaces
\begin{equation*}
  \mathrm{CE}(\mathcal{C}) = \Hom_{\mathcal{C}}(\mathbb{F}, \unitobj)
  \quad \text{and} \quad
  \mathrm{CF}(\mathcal{C}) = \Hom_{\mathcal{C}}(\unitobj, \mathbb{F})
\end{equation*}
associated to the coend $\mathbb{F} = \int^{X \in \mathcal{C}} X^* \otimes X$. The former can be thought of as the space of `central elements'. Indeed, we have
\begin{equation}
  \label{eq:CE-interpretation}
  \mathrm{CE}(\mathcal{C})
  \cong \int_{X \in \mathcal{C}} \Hom_{\mathcal{C}}(X^* \otimes X, \unitobj)
  \cong \int_{X \in \mathcal{C}} \Hom_{\mathcal{C}}(X, X)
  \cong \End(\id_{\mathcal{C}})
\end{equation}
by the basic properties of ends and coends \cite[IX]{MR1712872}. On the other hand, if $\mathcal{C} = \Rep(H)$ for some finite-dimensional Hopf algebra $H$, then we have
\begin{equation}
  \label{eq:CF-interpretation}
  \mathrm{CF}(\mathcal{C})
  \cong \{ f \in \Hom_k(H, k) \mid
  \text{$f(b a) = f(a S^2(b))$ for all $a, b \in H$} \},
\end{equation}
where $S$ is the antipode of $H$ (see the end of Section 3 of \cite{2015arXiv150401178S}). Thus $\mathrm{CF}(\mathcal{C})$ is an analogue of the space of `class functions'.

If $\mathcal{C}$ is braided, then the Hopf pairing $\omega: \mathbb{F} \otimes \mathbb{F} \to \unitobj$ induces a linear map
\begin{equation}
  \label{eq:def-Omega-C}
  \Omega_{\mathcal{C}}: \mathrm{CF}(\mathcal{C}) \to \mathrm{CE}(\mathcal{C}),
  \quad f \mapsto (f \otimes \id) \circ \omega.
\end{equation}
Let $H$ be a finite-dimensional quasitriangular Hopf algebra. By~\eqref{eq:CE-interpretation}, \eqref{eq:CF-interpretation} and the description of $\omega_{\Rep(H)}$ given in \cite[\S7.4.6]{MR1862634}, we see that $\Omega_{\Rep(H)}$ is bijective if and only if $H$ is {\em weakly-factorizable} in the sense of Takeuchi \cite[Definition 5.1]{MR1850651}. Based on the above discussion, we introduce the following terminology:

\begin{definition}
  A braided finite tensor category $\mathcal{C}$ is {\em weakly-factorizable} if the linear map $\Omega_{\mathcal{C}}$ is bijective.
\end{definition}

As observed in  Bakalov-Kirillov \cite{MR1797619} and Takeuchi \cite{MR1850651}, the map $\Omega_{\mathcal{C}}$ closely relates to the $S$-matrix in the semisimple case. Let $\mathcal{C}$ be a ribbon fusion category with braiding $\sigma$ and twist $\theta$, and let $\{ V_i \}_{i = 0}^m$ be the complete set of representatives of the isomorphism classes of simple objects with $V_0 = \unitobj$. For $i, j = 0, \dotsc, m$, the $(i, j)$-th entry $s_{i j}$ of the $S$-matrix is defined to be the quantum trace of
\begin{equation*}
  V_i \otimes V_j^*
  \xrightarrow{\quad \sigma \quad} V_j^* \otimes V_i
  \xrightarrow{\quad \sigma \quad} V_i \otimes V_j^*.
\end{equation*}
As $\mathcal{C}$ is semisimple, we may assume
\begin{equation}
  \label{eq:coend-ss}
  \mathbb{F} = \bigoplus_{i = 0}^m V_i^* \otimes V_i
\end{equation}
and the $X$-th component $i(\unitobj; X): X^* \otimes X \to \mathbb{F}$ of the universal dinatural transformation is just the inclusion morphism if $X$ is one of $V_0, \dotsc, V_m$. We set
\begin{equation*}
  \coev_X' = (\id_{X^*} \otimes \theta_X) \circ \sigma_{X,X^*} \circ \coev_X \quad (X \in \mathcal{C})
\end{equation*}
and then define $\chi_i \in \mathrm{CF}(\mathcal{C})$ and $e_i \in \mathrm{CF}(\mathcal{C})$ ($i = 0, \dotsc, m$) by
\begin{equation*}
  \chi_i: \unitobj
  \xrightarrow{\ \coev' \ }
  V_i^* \otimes V_i
  \xrightarrow{\ \text{inclusion} \ } \mathbb{F}
  \text{\quad and \quad}
  e_i: \mathbb{F}
  \xrightarrow{\ \text{projection} \ }
  V_i^* \otimes V_i
  \xrightarrow{\ \eval \ } \unitobj.
\end{equation*}
The sets $\{ \chi_i \}_{i = 0}^m$ and $\{ e_i \}_{i = 0}^m$ are bases of $\mathrm{CF}(\mathcal{C})$ and $\mathrm{CE}(\mathcal{C})$, respectively. With respect to these bases, the linear map $\Omega_{\mathcal{C}}$ is represented as
\begin{equation*}
  \Omega_{\mathcal{C}}(\chi_i) = \sum_{j = 0}^m \frac{s_{i j}}{s_{0 j}} e_j \quad (i = 0, \dotsc, m).
\end{equation*}
Consequently, the ribbon fusion category $\mathcal{C}$ is a modular tensor category if and only if $\mathcal{C}$ is weakly-factorizable.

For a general braided finite tensor category $\mathcal{C}$, the $S$-matrix of $\mathcal{C}$ cannot be defined at least in an obvious way. On the other hand, the map $\Omega_{\mathcal{C}}$ can be defined in the non-semisimple case. Thus it is natural to ask how the weak-factorizability relates to the non-degeneracy of $\mathcal{C}$. We answer this question by proving:

\begin{theorem}
  \label{thm:inj-tricen}
  If the map $\Omega_{\mathcal{C}}$ is injective, then the M\"uger center of $\mathcal{C}$ is trivial.
\end{theorem}

Combining this theorem with Theorems~\ref{thm:factor-nondege} and~\ref{thm:factor-tricen}, we see that a braided finite tensor category $\mathcal{C}$ is non-degenerate if and only if it is factorizable, if and only if it is weakly-factorizable, if and only if the map $\Omega_{\mathcal{C}}$ is injective, if and only if the M\"uger center of $\mathcal{C}$ is trivial.

\begin{remark}
  If $\mathcal{C}$ is {\em unimodular} in the sense of \cite{MR2097289}, then $\mathrm{CE}(\mathcal{C})$ and $\mathrm{CF}(\mathcal{C})$ have the same dimension \cite{2015arXiv150401178S}. Thus, in this case, $\Omega_{\mathcal{C}}$ is bijective if and only if it is surjective.
\end{remark}

\begin{remark}
  \label{rem:Sw-4-dim-Hopf}
  The surjectivity of the map $\Omega_{\mathcal{C}}$ does not imply the bijectivity of $\Omega_{\mathcal{C}}$ in general. Indeed, let $H$ be the algebra over the field $\mathbb{C}$ of complex numbers generated by $g$ and $x$ subject to the relations $g^2 = 1$, $x^2 = 0$ and $g x = - x g$. The algebra $H$ is a quasitriangular Hopf algebra (called Sweedler's four-dimensional Hopf algebra in literature) with the comultiplication $\Delta$, the counit $\varepsilon$, the antipode $S$ and the universal R-matrix $R$ determined by
  \begin{gather*}
    \Delta(g) = g \otimes g,
    \quad \Delta(x) = x \otimes g + 1 \otimes x,
    \quad \varepsilon(g) = 1,
    \quad \varepsilon(x) = 0, \\
    \quad S(g) = g,
    \quad S(x) = - g x,
    \quad R = \frac{1}{2} (1 \otimes 1 + 1 \otimes g + g \otimes 1 - g \otimes g).
  \end{gather*}
  The category $\mathcal{C} := \Rep(H)$ is a symmetric finite tensor category. Let $\alpha: H \to \mathbb{C}$ be the algebra map sending $g$ and $x$ to $-1$ and $0$, respectively. Then we have
  \begin{equation*}
    \mathrm{CE}(\mathcal{C})
    = \End(\id_{\mathcal{C}}) \cong \mathbb{C}
    \quad \text{and} \quad
    \mathrm{CF}(\mathcal{C})
    = \mathrm{span}_{\mathbb{C}} \{ \varepsilon, \alpha \}
  \end{equation*}
  under the identifications~\eqref{eq:CE-interpretation} and~\eqref{eq:CF-interpretation}. The map $\Omega_{\mathcal{C}}$ is surjective in this case, but not bijective.
\end{remark}

\subsection{Results on class functions}

Let $\mathcal{C}$ be a finite tensor category. We provide some results on the space $\mathrm{CF}(\mathcal{C})$ of `class functions' to prove Theorem~\ref{thm:inj-tricen}. Let $\mathcal{S}$ be a topologizing subcategory of $\mathcal{C}$. As we have mentioned in \S\ref{subsec:FPdim-subsec}, the coend
\begin{equation*}
  \mathbb{F}_{\mathcal{S}} = \int^{X \in \mathcal{S}} X^* \otimes X \in \mathcal{C}
\end{equation*}
exists. Let $i_{\mathcal{S}}(X): X^* \otimes X \to \mathbb{F}_{\mathcal{S}}$ ($X \in \mathcal{S}$) denote the universal dinatural transformation of the coend $\mathbb{F}_{\mathcal{S}}$. Note that $\mathbb{F} = \mathbb{F}_{\mathcal{C}}$. By the universal property of the coend $\mathbb{F}_{\mathcal{S}}$, there is a unique morphism $\phi_{\mathcal{S}}: \mathbb{F}_{\mathcal{S}} \to \mathbb{F}$ in $\mathcal{C}$ such that the equation
\begin{equation}
  \label{eq:coend-can-inc}
  i_{\mathcal{C}}(X) = \phi_{\mathcal{S}} \circ i_{\mathcal{S}}(X)
\end{equation}
holds for all objects $X \in \mathcal{S}$. We call $\phi_{\mathcal{S}}$ the {\em canonical inclusion} in view of the following lemma proved in \cite{2015arXiv150401178S}:

\begin{lemma}
  Let $\mathcal{S}$ be a topologizing subcategory of a finite tensor category $\mathcal{C}$. Then the morphism $\phi_{\mathcal{S}}$ defined by~\eqref{eq:coend-can-inc} is a monomorphism. Thus, in particular, the following linear map is injective:
  \begin{equation}
    \label{eq:cf-subcat-inc}
    \Hom_{\mathcal{C}}(\unitobj, \mathbb{F}_{\mathcal{S}})
    \xrightarrow{\quad \Hom_{\mathcal{C}}(\unitobj, \phi_{\mathcal{S}}) \quad}
    \Hom_{\mathcal{C}}(\unitobj, \mathbb{F})
    = \mathrm{CF}(\mathcal{C}).
  \end{equation}
\end{lemma}

The above lemma yields the following lower bound of the dimension of the space of class functions (we have considered the case where $\mathcal{C}$ is pivotal in \cite{2015arXiv150401178S}).

\begin{lemma}
  \label{lem:lower-bound-CF}
  Let $\{ V_i \}_{i = 0}^m$ be the complete set of representatives of the isomorphism classes of simple objects of a finite tensor category $\mathcal{C}$. Then we have
  \begin{equation}
    \label{ineq:lower-bound-CF}
    \# \{ i = 0, \dotsc, m \mid V_i \cong V_i^{**} \} \le \dim_k \mathrm{CF}(\mathcal{C}).
  \end{equation}
\end{lemma}
\begin{proof}
  Let $\mathcal{S}$ be the full subcategory of $\mathcal{C}$ consisting of semisimple objects. As $\mathcal{S}$ is semisimple, $\mathbb{F}_{\mathcal{S}}$ is of the form~\eqref{eq:coend-ss}. Thus we have
  \begin{equation*}
    \Hom_{\mathcal{C}}(\unitobj, \mathbb{F}_{\mathcal{S}})
    \cong \bigoplus_{i = 0}^m \Hom_{\mathcal{C}}(\unitobj, V_i^* \otimes V_i)
    \cong \bigoplus_{i = 0}^m \Hom_{\mathcal{C}}(V_i^{**}, V_i).
  \end{equation*}
  By Schur's lemma, $\dim_k \Hom_{\mathcal{C}}(\unitobj, \mathbb{F}_{\mathcal{S}})$ is equal to the left-hand side of~\eqref{ineq:lower-bound-CF}. The claim of this lemma now follows from the injectivity of the map \eqref{eq:cf-subcat-inc}.
\end{proof}

To prove Theorem~\ref{thm:inj-tricen}, we need the following technical lemma:

\begin{lemma}
  \label{lem:triviality-lemma}
  A finite tensor category $\mathcal{C}$ is equivalent to $\mathrm{Vec}$ if and only if
  \begin{equation}
    \label{eq:dim-CF-is-1}
    \dim_k \mathrm{CF}(\mathcal{C}) = 1.
  \end{equation}
\end{lemma}
\begin{proof}
  The `only if' part follows from \eqref{eq:CF-interpretation} with $H = k$. The `if' part is proved as follows: Suppose that \eqref{eq:dim-CF-is-1} holds. Let $D \in \mathcal{C}$ be the {\em distinguished invertible object} introduced in \cite{MR2097289} as a categorical analogue of the modular function. As it is invertible, $D$ is a simple object such that $D \cong D^{**}$. Thus, by Lemma~\ref{lem:lower-bound-CF}, we have $D \cong \unitobj$. This means that $\mathcal{C}$ is {\em unimodular} in the sense of  \cite{MR2097289}.

  We have introduced the notions of integrals and cointegrals for unimodular finite tensor categories in \cite{2015arXiv150401178S}. As an application of these notions, we have established the Maschke-type theorem for such categories \cite[Proposition 5.9]{2015arXiv150401178S}. The theorem implies that $\mathcal{C}$ is semisimple if and only if the algebra $\mathrm{CE}(\mathcal{C})$ ($\cong \End(\id_{\mathcal{C}})$) has no non-zero nilpotent elements. Now we have isomorphisms
  \begin{equation*}
    \mathrm{CE}(\mathcal{C}) \cong \mathrm{CF}(\mathcal{C}) \cong k
  \end{equation*}
  of vector spaces by \eqref{eq:dim-CF-is-1} and the Fourier transform \cite[Definition 5.11]{2015arXiv150401178S}. Thus $\mathrm{CE}(\mathcal{C})$ is isomorphic to $k$ as an algebra. By the above-mentioned fact, we conclude that $\mathcal{C}$ is semisimple.

  Since $\mathcal{C}$ is semisimple, every simple object $V \in \mathcal{C}$ satisfies $V \cong V^{**}$. Thus, again by Lemma~\ref{lem:lower-bound-CF}, we see that the unit object of $\mathcal{C}$ is the unique simple object of $\mathcal{C}$ (up to isomorphisms). This means that $\mathcal{C} \approx \mathrm{Vec}$. The proof is completed.
\end{proof}

\subsection{Proof of Theorem~\ref{thm:inj-tricen}}
\label{subsec:proof-thm-inj-tricen}

We give a proof of Theorem~\ref{thm:inj-tricen}. Let $\mathcal{C}$ be a braided finite tensor category, and set $\mathcal{S} = \mathcal{C}'$. Then we have
\begin{equation}
  \label{eq:omega-centralizer}
  \omega_{\mathcal{C}} \circ (\phi_{\mathcal{S}} \otimes \id_{\mathbb{F}}) = (\varepsilon \circ \phi_{\mathcal{S}}) \otimes \varepsilon,
\end{equation}
where $\varepsilon: \mathbb{F} \to \unitobj$ is the counit. Indeed, for all $X \in \mathcal{S}$ and $Y \in \mathcal{C}$, we have
\begin{align*}
  & \omega_{\mathcal{C}} \circ (\phi_{\mathcal{S}} \otimes \id_{\mathbb{F}}) \circ (i_{\mathcal{S}}(X) \otimes i_{\mathcal{C}}(Y)) \\
  & = (\eval_X \otimes \eval_Y) \circ (\id_{X^*} \otimes \sigma_{Y^*,X} \sigma_{X,Y^*} \otimes \id_{Y})
  & & \text{(by~\eqref{eq:F-def-omega},~\eqref{eq:coend-can-inc})} \\
  & = \phantom{(}\eval_X \otimes \eval_Y
  & & \text{(by the definition of $\mathcal{C}'$)} \\
  & = (\varepsilon \phi_{\mathcal{S}} \otimes \varepsilon) \circ (i_{\mathcal{S}}(X) \otimes i_{\mathcal{C}}(Y))
  & & \text{(by~\eqref{eq:F-def-coalg},~\eqref{eq:coend-can-inc})}.
\end{align*}
For $f \in \mathrm{CF}(\mathcal{S})$, we define $c(f) \in k$ by $c(f) \cdot \id_{\unitobj} = \varepsilon \circ \phi_{\mathcal{S}} \circ f$. By~\eqref{eq:omega-centralizer}, the map
\begin{equation}
  \label{eq:cf-subcat-inc-2}
  \mathrm{CF}(\mathcal{S})
  = \Hom_{\mathcal{C}}(\unitobj, \mathbb{F}_{\mathcal{S}})
  \xrightarrow{\quad \eqref{eq:cf-subcat-inc} \quad} \mathrm{CF}(\mathcal{C})
  \xrightarrow{\quad \Omega_{\mathcal{C}} \quad}
  \mathrm{CE}(\mathcal{C})
\end{equation}
sends $f \in \mathrm{CF}(\mathcal{S})$ to $c(f) \cdot \varepsilon$. If $f = u$ is the unit of $\mathbb{F}_{\mathcal{S}}$, then $c(f) = 1$. Hence the image of $\mathrm{CF}(\mathcal{S})$ under~\eqref{eq:cf-subcat-inc-2} is a one-dimensional subspace spanned by $\varepsilon$.

Now suppose that $\Omega_{\mathcal{C}}$ is injective. Then \eqref{eq:cf-subcat-inc-2} is also injective as the composition of injective maps. Hence, by the above argument, $\mathrm{CF}(\mathcal{S})$ is one-dimensional. Thus we apply Lemma~\ref{lem:triviality-lemma} to obtain $\mathcal{C}' = \mathcal{S} \approx \mathrm{Vec}$. The proof is done.

\subsection{On the rank of $\Omega_{\mathcal{C}}$}

Let $\mathcal{C}$ be a braided finite tensor category, and let $\Omega_{\mathcal{C}}$ be the linear map defined by~\eqref{eq:def-Omega-C}. In view of our results, the rank of $\Omega_{\mathcal{C}}$ seems to be an important invariant of $\mathcal{C}$. We consider the subspace
\begin{equation*}
  \Xi_{\mathcal{C}} := \{ \xi \in \End(\id_{\mathcal{C}})
  \mid \text{$\xi_{V \otimes X} = \xi_V \otimes \id_X$ for all $V \in \mathcal{C}$ and $X \in \mathcal{C}'$} \}
\end{equation*}
of $\End(\id_{\mathcal{C}})$. Then the rank of $\Omega_{\mathcal{C}}$ is bounded from the above as follows:

\begin{proposition}
  \label{prop:rank-Omega}
  $\mathrm{rank} \, \Omega_{\mathcal{C}} \le \dim_k \Xi_{\mathcal{C}}$.
\end{proposition}
\begin{proof}
  We prove the claim by showing that the image of the linear map
  \begin{equation}
    \label{eq:rank-Omega-pf-1}
    \mathrm{CF}(\mathcal{C})
    \xrightarrow{\quad \Omega_{\mathcal{C}} \quad}
    \mathrm{CE}(\mathcal{C})
    \xrightarrow[\cong]{\quad \eqref{eq:CE-interpretation} \quad}
    \End(\id_{\mathcal{C}})
  \end{equation}
  is contained in $\Xi_{\mathcal{C}}$. Let $\xi^{(f)} \in \End(\id_{\mathcal{C}})$ be the image of $f \in \mathrm{CF}(\mathcal{C})$ under~\eqref{eq:rank-Omega-pf-1}. By definition, $\xi^{(f)}$ is the natural transformation given by
  \begin{equation*}
    \xi^{(f)}_{V} = (\id_{V} \otimes \Omega_{\mathcal{C}}(f)) \circ (\id_V \otimes i(V)) \circ (\coev_V \otimes \id_V)
    \quad (V \in \mathcal{C}),
  \end{equation*}
  where $i(V): V^* \otimes V \to \mathbb{F}$ ($V \in \mathcal{C}$) is the universal dinatural transformation for the coend $\mathbb{F}$. We prove the equation
  \begin{equation}
    \label{eq:rank-Omega-pf-5}
    \xi_{V \otimes X}^{(f)} = \xi_V^{(f)} \otimes \id_X
    \quad (V \in \mathcal{C}, X \in \mathcal{C}')
  \end{equation}
  as in Figure~\ref{fig:proof-rank-Omega-pf-5}. The first and the last diagrams in the figure express the left-hand side and the right-hand side of~\eqref{eq:rank-Omega-pf-5}, respectively. The first and the third equalities of Figure~\ref{fig:proof-rank-Omega-pf-5} follow from the definition of the structure morphisms of $\mathbb{F}$. To verify the second equality of Figure~\ref{fig:proof-rank-Omega-pf-5}, we prove
  \begin{equation}
    \label{eq:rank-Omega-pf-4}
    \omega \circ (\id_{\mathbb{F}} \otimes i(X)) = (\varepsilon \otimes \varepsilon) \circ (\id_{\mathbb{F}} \otimes i(X))
  \end{equation}
  in a similar way as~\eqref{eq:omega-centralizer}. Now the equality is verified as follows:
  \begin{align*}
    & \omega \circ (\id_{\mathbb{F}} \otimes m) \circ (\id_{\mathbb{F}} \otimes i(V) \otimes i(X)) \\
    & = \omega
      \circ (\id_{\mathbb{F}} \otimes \omega \otimes \id_{\mathbb{F}})
      \circ (\Delta \otimes i(V) \otimes i(X))
    & & \text{(by~\eqref{eq:F-Hopf-pairing-2})} \\
    & = (\varepsilon \otimes \varepsilon)
      \circ (\id_{\mathbb{F}} \otimes \omega \otimes \id_{\mathbb{F}})
      \circ (\Delta \otimes i(V) \otimes i(X))
    & & \text{(by~\eqref{eq:rank-Omega-pf-4})} \\
    & = (\omega \otimes \varepsilon) \circ (\id_{\mathbb{F}} \otimes i(V) \otimes i(X)).
  \end{align*}
  Hence $\xi^{(f)} \in \Xi_{\mathcal{C}}$ for all $f \in \mathrm{CF}(\mathcal{C})$. The proof is done.
\end{proof}

\begin{figure}
  \input{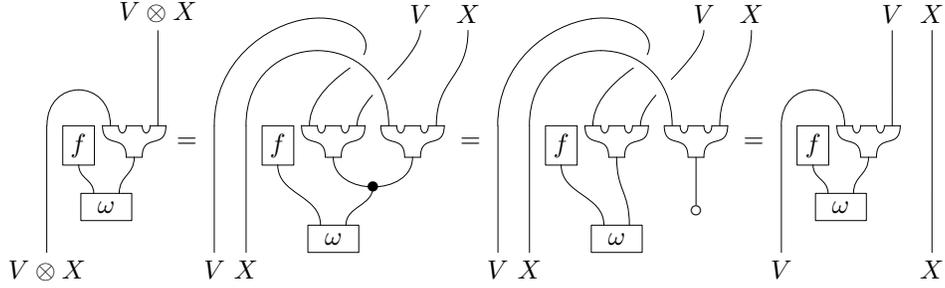}
  \caption{Proof of Equation~\eqref{eq:rank-Omega-pf-5}}
  \label{fig:proof-rank-Omega-pf-5}
\end{figure}

\begin{remark}
  Proposition~\ref{prop:rank-Omega} is motivated by \cite[Theorem 3.4]{MR2609644}, which states that the rank of the $S$-matrix of a ribbon fusion category $\mathcal{F}$ is equal to the number of $\mathcal{F}'$-components of $\mathcal{F}$. In view of \cite[Theorem 3.4]{MR2609644}, we might conjecture that the image of \eqref{eq:rank-Omega-pf-1} is precisely $\Xi_{\mathcal{C}}$, and thus the equation
  \begin{equation}
    \label{eq:rank-Omega-eq}
    \mathrm{rank} \, \Omega_{\mathcal{C}} = \dim_k \Xi_{\mathcal{C}}
  \end{equation}
  holds. Our results show that~\eqref{eq:rank-Omega-eq} holds if $\mathcal{C}$ is non-degenerate (but the converse does not by Remark~\ref{rem:Sw-4-dim-Hopf}). Our arguments in \S\ref{subsec:proof-thm-inj-tricen} imply that \eqref{eq:rank-Omega-eq} holds if $\mathcal{C}$ is symmetric.
\end{remark}

\section{Non-degeneracy of the Yetter-Drinfeld category}
\label{sec:non-dege-YD}

\subsection{The Yetter-Drinfeld category}

Let $\mathcal{C}$ be a braided finite tensor category with braiding $\sigma$, and let $B$ be a Hopf algebra in $\mathcal{C}$ with multiplication $m$, unit $u$, comultiplication $\Delta$, counit $\varepsilon$ and antipode $S$ (note that $S$ is invertible in this case \cite{MR1685417}). The structure morphisms of $B$ will be depicted as in Figure~\ref{fig:gra-cal} of \S\ref{subsec:rep-th-of-F}.
For $M \in \mathcal{C}_B$ and $N \in \mathcal{C}^B$, the action of $M$ and the coaction of $N$ are illustrated as in Figure~\ref{fig:gra-cal-act}.

\begin{definition}
  A (right-right) Yetter-Drinfeld module over $B$ ($=$ a crossed module of Bespalov \cite{MR1456522}) is an object $M \in \mathcal{C}$
  endowed with a right action $\triangleleft_M$ and a right coaction $\delta_M$ of $B$ satisfying the Yetter-Drinfeld condition
  \begin{equation}
    \label{eq:YD-cond}
    \begin{gathered}
      (\id_M \otimes m) \circ (\sigma_{B,M} \otimes \id_B) \circ (\id_B \otimes \delta_M \triangleleft_M) \circ (\sigma_{M,B} \otimes \id_B) \circ (\id_M \otimes \Delta) \\
      = (\triangleleft_M \otimes m) \circ (\id_M \otimes \sigma_{M,B} \otimes \id_B) \circ (\delta_M \otimes \Delta),
    \end{gathered}
  \end{equation}
  which is illustrated in Figure~\ref{fig:YD-cond}. A morphism of Yetter-Drinfeld modules over $B$ is a morphism between underlying objects that is both $B$-linear and $B$-colinear. We denote by $\mathcal{YD}(\mathcal{C})^B_B$ the category of the Yetter-Drinfeld modules over $B$ and call it the {\em Yetter-Drinfeld category} of $B$.
\end{definition}

\begin{figure}
  \input{fig-S6-1}
  \caption{Graphical notation for (co)actions}
  \label{fig:gra-cal-act}    

  \bigskip
  \input{fig-S6-2}
  \caption{Yetter-Drinfeld condition}
  \label{fig:YD-cond}

  \bigskip
  \input{fig-S6-3}
  \caption{The braiding of the Yetter-Drinfeld category}
  \label{fig:YD-braiding}
\end{figure}

When the action $\triangleleft_M$ and the coaction $\delta_M$ are trivial, the Yetter-Drinfeld condition~\eqref{eq:YD-cond} reduces to $\sigma_{B,M} \circ \sigma_{M,B} = \id_{M \otimes B}$ (see also Figure~\ref{fig:YD-cond}).
Thus an object of the M\"uger center $\mathcal{C}'$ turns into a Yetter-Drinfeld module over $B$ by the trivial action and the trivial coaction. We can regard $\mathcal{C}'$ as a full subcategory of $\mathcal{YD}(\mathcal{C})^B_B$ in this way.

As shown by Bespalov \cite[Section 3]{MR1456522}, the category $\mathcal{YD}(\mathcal{C})^B_B$ is a rigid monoidal category with the monoidal structure inherited from $\mathcal{C}_B$ and $\mathcal{C}^B$. Moreover, it has the braiding $\Sigma$ given by
\begin{equation}
  \label{eq:YD-braiding}
  \Sigma_{M, N} = (\id_N \otimes \triangleleft_M) \circ (\sigma_{M, N} \otimes \id_B) \circ (\id_M \otimes \delta_N)
\end{equation}
for $M, N \in \mathcal{YD}(\mathcal{C})^B_B$, where $\triangleleft_M: M \otimes B \to M$ and $\delta_N: N \to N \otimes B$ are the action and the coaction of $B$ on $M$ and $N$, respectively.
The braiding $\Sigma$ and its inverse are expressed as in Figure~\ref{fig:YD-braiding}.

The main result of this section is:

\begin{theorem}
  \label{thm:YD-br-FTC}
  The braided monoidal category $\mathcal{YD}(\mathcal{C})^B_B$ is in fact a braided finite tensor category with Frobenius-Perron dimension
  \begin{equation*}
    \FPdim(\mathcal{YD}(\mathcal{C})^B_B) = \FPdim(\mathcal{C}) \FPdim(B)^2.
  \end{equation*}
  The M\"uger center of $\mathcal{YD}(\mathcal{C})^B_B$ is given by
  \begin{equation*}
    (\mathcal{YD}(\mathcal{C})^B_B)' = \mathcal{C}',
  \end{equation*}
  where $\mathcal{C}'$ is regarded as a full subcategory of $\mathcal{YD}(\mathcal{C})^B_B$ in the above way.
\end{theorem}

Thus, by Theorems~\ref{thm:factor-nondege} and~\ref{thm:factor-tricen}, the braided finite tensor category $\mathcal{YD}(\mathcal{C})^B_B$ is non-degenerate if and only if $\mathcal{C}$ is. There are also braided monoidal categories ${}_B^B \mathcal{YD}(\mathcal{C})$, ${}_B\mathcal{YD}(\mathcal{C})^B$ and ${}^B\mathcal{YD}(\mathcal{C})_B$ of left-left, left-right, right-left Yetter-Drinfeld modules over $B$. Since these categories are isomorphic to $\mathcal{YD}(\mathcal{C})^B_B$ as braided monoidal categories \cite[Corollary 3.5.5]{MR1456522}, they are also braided finite tensor categories, which are non-degenerate precisely if $\mathcal{C}$ is.

\subsection{The fundamental theorem for Hopf bimodules}

To prove the main theorem of this section, we first recall from \cite{MR1492897} the fundamental theorem for Hopf bimodules 
over a braided Hopf algebra.

Let $\mathcal{C}$ be a braided finite tensor category, and let $B \in \mathcal{C}$ be a Hopf algebra.
The category ${}_B \mathcal{C}_B$ of $B$-bimodules in $\mathcal{C}$ is a finite tensor category with the monoidal structure inherited from ${}_B \mathcal{C}$ and $\mathcal{C}_B$, and the triple $(B, \Delta, \varepsilon)$ is in fact a coalgebra in ${}_B \mathcal{C}_B$.
Thus we can consider the category
\begin{equation*}
  {}^B_B \mathcal{C}_B^{} := {}^B({}_B^{}\mathcal{C}_B^{})
\end{equation*}
of left $B$-comodules in ${}_B\mathcal{C}_B$. If $M \in \mathcal{C}_B$, then $B \otimes M$ is an object of ${}^B_B\mathcal{C}^{}_B$ by the left action $\triangleright$, the right action $\triangleleft$ and the left coaction $\delta$ given respectively by
\begin{equation*}
  \triangleright = m \otimes \id_M,
  \quad \triangleleft = (m \otimes \id_M) \circ (\id_B \otimes \sigma_{M, B})
  \quad \text{and} \quad
  \delta = (\Delta \otimes \id_M).
\end{equation*}
Let $B \ltimes M$ denote the object of ${}_B^{B} \mathcal{C}_B^{}$ obtained from $M \in \mathcal{C}_B$ in this way.
Bespalov and Drabant \cite[Proposition 3.6.3]{MR1492897} showed that the functor
\begin{equation}
  \label{eq:Hopf-bimod-eq-1}
  \mathcal{C}_B
  \xrightarrow{\quad \approx \quad}
  {}^B_B \mathcal{C}_B^{},
  \quad M \mapsto  B \ltimes M
\end{equation}
is an equivalence (the fundamental theorem for Hopf bimodules). 
Let ${}^B_B\mathcal{C}^B_B$ be the category of $B$-bicomodules in ${}_B \mathcal{C}_B$, and let
\begin{equation*}
  U_{1}: \mathcal{YD}(\mathcal{C})^B_B \to \mathcal{C}_B^{}
  \quad \text{and} \quad
  U_{2}: {}^B_B \mathcal{C}^B_B \to {}^B_B \mathcal{C}_B^{}
\end{equation*}
be the forgetful functors. 
Bespalov and Drabant \cite[Theorem 4.3.2]{MR1492897} also showed that the equivalence \eqref{eq:Hopf-bimod-eq-1} lifts to an equivalence
\begin{equation}
  \label{eq:Hopf-bimod-eq-2}
  \mathcal{YD}(\mathcal{C})^B_B
  \xrightarrow{\quad \approx \quad}
  {}^B_B\mathcal{C}^B_B
\end{equation}
of categories such that the following diagram is commutative:
\begin{equation}
  \label{eq:Hopf-bimod-eq-3}
  \xymatrix{
    \mathcal{YD}(\mathcal{C})^B_B \ar[d]_{U_1}
    \ar[rr]^{\text{\eqref{eq:Hopf-bimod-eq-2}}}_{\approx}
    & & {}^B_B \mathcal{C}^B_B\phantom{,}
    \ar[d]^{U_2} \\
    \mathcal{C}_B
    \ar[rr]^{\text{\eqref{eq:Hopf-bimod-eq-1}}}_{\approx}
    & & {}^B_B \mathcal{C}^{}_B.
  }
\end{equation}

\subsection{Finiteness of the Yetter-Drinfeld category}

We keep the notations of the previous subsection.
Now we prove that the Yetter-Drinfeld category $\mathcal{YD}(\mathcal{C})^B_B$ is a braided finite tensor category.
The functor $U_2$ has a right adjoint given by
\begin{equation*}
  R_2: {}_B^B\mathcal{C}_B^{} \to {}^B_B \mathcal{C}^B_B,
  \quad M \mapsto M \otimes B.
\end{equation*}
By the above commutative diagram, $U_1$ also has a right adjoint, say $R_1$.
It is easy to see that $U_2 \circ R_2$ is a $k$-linear exact comonad on ${}^B_B\mathcal{C}_B^{}$ whose category of comodules is precisely ${}^B_B\mathcal{C}^B_B$. Thus, without knowing any explicit description of $R_1$,
we see that $U_1 \circ R_1$ is a $k$-linear exact comonad on $\mathcal{C}_B$
whose category of comodules is equivalent to $\mathcal{YD}(\mathcal{C})^B_B$.
This proves:

\begin{lemma}
  $\mathcal{YD}(\mathcal{C})^B_B$ is a finite tensor category.
\end{lemma}
\begin{proof}
  By Lemma~\ref{lem:finiteness-lemma} and the above argument, $\mathcal{YD}(\mathcal{C})^B_B$ is a finite abelian category. The other axioms of finite tensor categories are verified easily.
\end{proof}

Next, we compute the Frobenius-Perron dimension of $\mathcal{YD}(\mathcal{C})^B_B$.
The commutative diagram~\eqref{eq:Hopf-bimod-eq-3} is also useful for this purpose.
Indeed, by~\eqref{eq:Hopf-bimod-eq-3}, we have
\begin{equation*}
  B \ltimes R_1(M) \cong (B \ltimes M) \otimes B
\end{equation*}
for all $M \in \mathcal{C}_B$. In particular, $B \otimes R_1(M) \cong B \otimes M \otimes B$ as objects of $\mathcal{C}$.
Thus the Frobenius-Perron dimension of $R_1(M)$ is given by
\begin{equation}
  \label{eq:FPdim-YD-ind}
  \FPdim(R_1(M)) = \FPdim(B) \FPdim(M).
\end{equation}
By~\eqref{eq:FPdim-ind} and~\eqref{eq:FPdim-YD-ind}, we have:

\begin{lemma}
  $\FPdim (\mathcal{YD}(\mathcal{C})^B_B) = \FPdim(\mathcal{C}) \FPdim(B)^2$.
\end{lemma}

For later use, we note:

\begin{lemma}
  \label{lem:YD-dominant}
  The forgetful functor $U: \mathcal{YD}(\mathcal{C})_B^B \to \mathcal{C}$ is dominant.
\end{lemma}
\begin{proof}
  Let $U_{B}: \mathcal{C}_B \to \mathcal{C}$ be the forgetful functor. Then $U$ is decomposed as:
  \begin{equation*}
    U: \mathcal{YD}(\mathcal{C})_B^B
    \xrightarrow{\quad U_2 \quad} \mathcal{C}_B
    \xrightarrow{\quad U_B \quad} \mathcal{C}.
  \end{equation*}
  If $X$ is a non-zero object of a finite tensor category, then the endofunctors $X \otimes (-)$ and $(-) \otimes X$ on the finite tensor category are faithful. Thus $R_2$ is faithful. By the commutative diagram~\eqref{eq:FPdim-YD-ind}, $R_1$ is also faithful, and hence the tensor functor $U_2$ is dominant. Similarly, since a left adjoint of $U_B$ ({\it i.e.}, the free $B$-module functor) is faithful, the tensor functor $U_B$ is dominant. Thus $U$ is also dominant as the composition of $k$-linear left exact dominant functors.
\end{proof}

\subsection{The M\"uger center of the Yetter-Drinfeld category}

To complete the proof of Theorem~\ref{thm:YD-br-FTC}, we compute the M\"uger center of $\mathcal{YD}(\mathcal{C})^B_B$.
If $M \in \mathcal{C}'$, then
\begin{equation*}
  \Sigma_{X,M} \circ \Sigma_{M,X} = \sigma_{X,M} \circ \sigma_{M,X} = \id_{M \otimes X}
\end{equation*}
for all $X \in \mathcal{YD}(\mathcal{C})^B_B$. Thus we have
\begin{equation*}
  \mathcal{C}' \subset (\mathcal{YD}(\mathcal{C})^B_B)'.
\end{equation*}
To prove the converse inclusion, we recall from \cite[Lemma 3.9.3]{MR1456522} that there are the Yetter-Drinfeld $B$-modules $P$ and $Q$ defined as follows:
\begin{enumerate}
\item The Yetter-Drinfeld module $P$ is the object $B \in \mathcal{C}$ with the coaction given by the comultiplication.
  The action $\triangleleft_{\mathrm{ad}}$ of $B$ on $P$ is given by
  \begin{equation*}
    \triangleleft_{\mathrm{ad}}
    = m \circ (S \otimes m) \circ (\sigma_{B,B} \otimes \id_B) \circ (\id_B \otimes \Delta).
  \end{equation*}
\item The Yetter-Drinfeld module $Q$ is the object $B \in \mathcal{C}$ with the action given by the multiplication.
  The coaction $\delta_{\mathrm{ad}}$ of $B$ on $Q$ is given by
  \begin{equation*}
    \delta_{\mathrm{ad}}
    = (\id_B \otimes m) \circ (\sigma_{B,B} \otimes \id_B) \circ (S \otimes \Delta) \circ \Delta.
  \end{equation*}
\end{enumerate}
Observe that the following equations hold:
\begin{equation*}
  \varepsilon \circ \triangleleft_{\mathrm{ad}} = \varepsilon \otimes \varepsilon
  \quad \text{and} \quad
  \delta_{\mathrm{ad}} \circ u = u \otimes u.
\end{equation*}
Now let $M$ be a Yetter-Drinfeld module over $B$ with action $\triangleleft_{M}$ and coaction $\delta_M$.
By the definition of the braiding $\Sigma$, we have the following four equations:
\begin{align*}
  \Sigma_{Q, M} \circ (u \otimes \id_M) & = \delta_M, &
  (\varepsilon \otimes \id_B) \circ \Sigma_{M,P} & = \triangleleft_M, \\
  \Sigma_{M, Q}^{-1} \circ (u \otimes \id_M) & = \id_M \otimes u, &
  (\varepsilon \otimes \id_B) \circ \Sigma_{P,M}^{-1} & = \id_M \otimes \varepsilon.
\end{align*}
Suppose that $M \in (\mathcal{YD}(\mathcal{C})^B_B)'$. Then we have
\begin{equation*}
  \delta_M
  = \Sigma_{Q, M} \circ (u \otimes \id_M)
  = \Sigma_{M, Q}^{-1} \circ (u \otimes \id_M) = \id_M \otimes u
\end{equation*}
and, in a similar way, $\triangleleft_M = \id_M \otimes \varepsilon$.
By the definition of the braiding $\Sigma$, we have $\Sigma_{N,M} = \sigma_{N,M}$ and $\Sigma_{M,N} = \sigma_{M,N}$ for all $N \in \mathcal{YD}(\mathcal{C})^B_B$. The assumption that $M$ belongs to the M\"uger center turns into:
\begin{equation*}
  \sigma_{N,M} \sigma_{M,N} = \id_{M \otimes N}
  \text{\quad for all $N \in \mathcal{YD}(\mathcal{C})^B_B$}.
\end{equation*}
By Lemma~\ref{lem:YD-dominant} and the naturality of the braiding $\sigma$, we see that this equation actually holds for all objects $N \in \mathcal{C}$.
Hence $M \in \mathcal{C}'$.
This shows
\begin{equation*}
  (\mathcal{YD}(\mathcal{C})^B_B)' \subset \mathcal{C}'.
\end{equation*}
The proof is done.

\subsection{Factorizable Hopf algebras}
\label{sec:fact-hopf-alg}

There are some constructions of Hopf algebras with triangular decomposition (such as the quantized universal enveloping algebra of a semsimple Lie algebra) from a Hopf algebra in a braided monoidal category \cite{MR1045735,MR1645545,MR1383798}. Combining such constructions with our result, we obtain a source of factorizable Hopf algebras: Let $H$ be a finite-dimensional factorizable Hopf algebra, and let $B$ be a Hopf algebra in $\mathcal{C} = \Rep(H)$. Since ${}^B\mathcal{YD}(\mathcal{C})_B$ has a fiber functor, there is a unique (up to isomorphism) quasitriangular Hopf algebra $U$ such that there is a $k$-linear braided monoidal equivalence
\begin{equation*}
  {}^B\mathcal{YD}(\mathcal{C})_B \approx \Rep(U)
\end{equation*}
commuting with the fiber functors. The quasitriangular Hopf algebra $U$ is given explicitly by the double-bosonization of Majid; see \cite[Appendix B]{MR1645545}. Our results guarantee that $U$ is factorizable. We note that the dimension of the Hopf algebra $U$ may not be a square number, since there is the decomposition
\begin{equation*}
  U = B^* \otimes_k H \otimes_k B
\end{equation*}
of vector spaces by the construction. Thus $U$ cannot be obtained by the Drinfeld double construction in general.

Sommerh\"auser \cite[Section 5]{MR1383798} and Majid \cite[Section 4]{MR1645545} have essentially pointed out that this kind of construction produces the so-called small quantum group $u_q(\mathfrak{g})$. Example~\ref{ex:small-quantum} below gives the detail of the construction with the help of the theory of Nichols algebras \cite{MR1913436,MR2630042}. Our results explain that $u_q(\mathfrak{g})$ with specific $q$ is factorizable (see Lyubashenko \cite{MR1354257} for more general case).

\begin{example}
  \label{ex:small-quantum}
  For simplicity, we assume $k = \mathbb{C}$. Let $A = (a_{i j})_{i, j = 1, \dotsc, m}$ be a Cartan matrix of finite type, and let $D = \mathrm{diag}(d_1, \dotsc, d_m)$ be a diagonal matrix such that $d_i \in \{1, 2, 3\}$ and $d_i a_{i j} = d_j a_{j i}$ for all $i, j$. We fix an odd integer $N > 1$ that is relatively prime to the determinant of $D A$ and consider the group
  \begin{equation*}
    \Gamma := \langle g_1, \dotsc, g_m \mid
    \text{$g_i^N = 1$, $g_i g_j = g_j g_i$ ($i, j = 1, \dotsc, m$)} \rangle.
  \end{equation*}
  Set $q := \exp(2 \pi \sqrt{-1} / N)$. There is the pairing $\langle \, , \, \rangle$ on $\Gamma$ given by
  \begin{equation}
    \label{eq:Gamma-pairing}
    \langle g_1^{i_1} \dotsb g_{m}^{i_m}, g_1^{j_1} \dotsb g_{m}^{j_m} \rangle
    = q^{i_1 j_1 + \dotsb + i_m j_m}
    \quad (i_1, \dotsc, i_m, j_1, \dotsc, j_m \in \mathbb{Z}).
  \end{equation}
  We define the bimultiplicative map $\chi: \Gamma \times \Gamma \to \mathbb{C}^{\times}$ by
  \begin{equation*}
    \chi(g_i, g_j) = q^{d_i a_{i j}}
    \quad (i, j = 1, \dotsc, m).
  \end{equation*}
  Let $H = \mathbb{C}\Gamma$ be the group algebra of $\Gamma$. The category $\Rep(H)$ is equivalent to the category of finite-dimensional $\Gamma$-graded vector spaces by the pairing~\eqref{eq:Gamma-pairing}. Braidings of such categories have been well-understood; see, {\it e.g.}, \cite[\S2.11]{MR2609644}. One can check that $H$ is a factorizable Hopf algebra with the universal R-matrix
  \begin{equation*}
    R = \sum_{g, h \in \Gamma} \chi(g, h) \, e_g \otimes e_h,
    \quad \text{where} \quad
    e_g = \sum_{x \in \Gamma} \overline{\langle g, x \rangle} x.
  \end{equation*}
  There is a $\mathbb{C}\Gamma$-module $V = \mathrm{span}_{\mathbb{C}} \{ x_1, \dotsc, x_m \}$ such that
  \begin{equation*}
    g \cdot x_i = \langle g, g_i \rangle x_i
    \quad (g \in \Gamma, i = 1, \dotsc, m).
  \end{equation*}
  By \cite[Theorem 4.3]{MR1913436}, the Nichols algebra $\mathfrak{B}(V)$ of $V$ is isomorphic to the positive part of the small quantum group $u_q(\mathfrak{g})$ associated to the semisimple Lie algebra $\mathfrak{g}$ corresponding to $A$. If $B = \mathfrak{B}(V)$, then the Hopf algebra $U$ in the above is shown to be isomorphic to $u_q(\mathfrak{g})$ by a similar computation as \cite{MR1645545}.
\end{example}

\bibliographystyle{alpha}

\def\cprime{$'$}

\end{document}